\documentclass[12pt, reqno]{amsart}
\usepackage{amsmath,amsopn,amssymb,amsthm}
\usepackage[english]{babel}
\usepackage[backref=page, breaklinks=true,colorlinks=true,linkcolor=blue,citecolor=blue,urlcolor=blue]{hyperref}

\usepackage{graphicx}

\renewenvironment{proof}[1][Proof]{\textbf{#1.} }{\ \rule{0.5em}{0.5em}}

\DeclareMathOperator{\bd}{bd}
\DeclareMathOperator{\co}{co}

\renewenvironment{proof}[1][Proof]{\textbf{#1.} }
{\ \rule{0.5em}{0.5em}}
\newtheorem{theorem}{Theorem}
\newtheorem{prop}{Proposition}

\newtheorem{conjecture}{Conjecture}

\theoremstyle{definition}

\newtheorem{remark}{Remark}

\newtheorem{problem}{Problem}

\begin{document}

\title
[Some extremal problems for polygons in the Euclidean plane]
{Some extremal problems for polygons \\ in the Euclidean plane}
\author{Yu.G.~Nikonorov, O.Yu.~Nikonorova}

\address{Nikonorov Yurii Gennadievich \newline
Southern Mathematical Institute of \newline
the Vladikavkaz Scientific Center of \newline
the Russian Academy of Sciences, \newline
Vladikavkaz, Markus st., 22, \newline
362027, Russia}
\email{nikonorov2006@mail.ru}

\address{Nikonorova Olga Yuryevna \newline
Bauman Moscow State Technical University, \newline
Moscow, 2-nd Baumanskaya, 5, \newline
105005, Russia}
\email{olyanik2003@gmail.com}

\begin{abstract}
The paper is devoted to some extremal problems, related to convex polygons in the Euclidean plane and their perimeters.
We present a number of results that have simple formulations, but rather intricate proofs. Related and still unsolved problems are discussed too.

\vspace{2mm}
\noindent
2020 Mathematical Subject Classification:
52A10, 52A40, 52B60.

\vspace{2mm} \noindent Key words and phrases:  convex figure, convex polygon, perimeter.
\end{abstract}

\maketitle

\section{Introduction}\label{sect.0}

In this paper, we deal with polygons in the Euclidean plane.
We identify the Euclidean plane with $\mathbb{R}^2$ supplied with the standard Euclidean metric~$d$, where $d(x,y)=\sqrt{(x_1-y_1)^2+(x_2-y_2)^2}$.
For any subset $A\subset \mathbb{R}^2$, $\co (A)$ means the {\it convex hull} of $A$. For every points $B,C \in \mathbb{R}^2$, $[B,C]$
denotes the line segment between these points.
In what follows, $O$ means the origin in $\mathbb{R}^2$, the symbols $B(x,\rho)$ and $S(x,\rho)$ denote, respectively, a closed ball (or disk) and a circle in $\mathbb{R}^2$
with center $x \in  \mathbb{R}^2$ and radius $\rho \geq 0$.

{\it A convex {\rm(}planar{\rm)} figure} is any compact convex subset of $\mathbb{R}^2$.
We shall denote by $L(K)$ and $\bd (K)$ the {\it perimeter} and the {\it boundary} of a convex figure $K$, respectively.
Recall that a {\it convex polygon} is the convex hull of a finite set of points in $\mathbb{R}^2$.
Since we will work only with convex polygons in the Euclidean plane, we will sometimes refer to them simply as {\it polygons} for brevity.
As references on convex geometry, one can advise, e.~g., \cite{Ball1997, BoFe1987, Had1957, MaMoOl}.
\smallskip

{\it A side $AB$ of a convex polygon $P$ is said to have an $l$-strut} if there exists a point $C\in P$ such that
such that $d(A,C)=d(B,C) = l$. It should be noted that $C$ should not be a vertex of $P$ in this definition.
For $l=1$ we will use the term {\it strut} instead of $1$-strut.

A convex polygon $P$ in Euclidean plane is said to have {\it the $\Delta(l)$ property} if
every its side has an $l$-strut (i.~e., for the endpoints $A$ and $B$ of every side of $P$, there is a point $C\in P$ such that $d(A,C)=d(B,C) = l$).
Let us recall Problem 1 from \cite{BMNN2021}, that is related to the study of self Chebyshev radii for boundaries of convex figures
(this geometric topic appeared due to \cite{Walter2017}).

\begin{problem}[\cite{BMNN2021}]
Given a real number $l>0$ and a natural number $n \geq 3$, determine the best possible constant $C(n,l)$ such that the inequality
$L(P) \geq C(n,l)$ holds for the perimeter $L(P)$ of every convex polygon $P$ with $n$ vertices that satisfies the $\Delta(l)$ property.
\end{problem}

It is easy to show that $C(3,l)=3l$. On the other hand, until now the answer for $n\geq 4$ was unknown.
One of the goals of this paper is the solution of the above problem. It is clear that, using similarities in the Euclidean plane, we may restrict our attention to the case $l=1$.
In what follows, we consider only the case $l=1$ and, for brevity, we will call the $\Delta(1)$ property {\it the $\Delta$ property}.
Our first main result is the following.

\begin{theorem}\label{th.minim}
Given a natural number $n \geq 3$. The perimeter $L(P)$ of any polygon~$P$ with $n$ vertices, that satisfies
the $\Delta$ property, satisfies $L(P)\geq 3$. Moreover, this inequality cannot be improved, while the equality $L(P)= 3$ holds
if and only if $P$ is a regular triangle with unit side.
\end{theorem}

At first, let us discuss the simplest case $n=3$.
Taking into account that the perimeter of any triangle that lies inside a given triangle $ABC$ (that satisfy the $\Delta$ property)
is not greater than
the perimeter of $\triangle ABC$ (see Proposition \ref{monotper} below), we see that $a+b =|BC|+|AC| \geq |BD|+|AD]=2$, see Fig. \ref{Fig_uniq}~a).
By analogy, we get $a+c\geq 2$ and $b+c \geq 2$.
Summing up these three inequalities, we get $a+b+c\geq 3$. This proves Theorem \ref{th.minim} for $n=3$.

\begin{figure}[t]
\begin{minipage}[h]{0.48\textwidth}
\center{\includegraphics[width=0.9\textwidth, trim=0mm 0mm 0mm 0mm, clip]{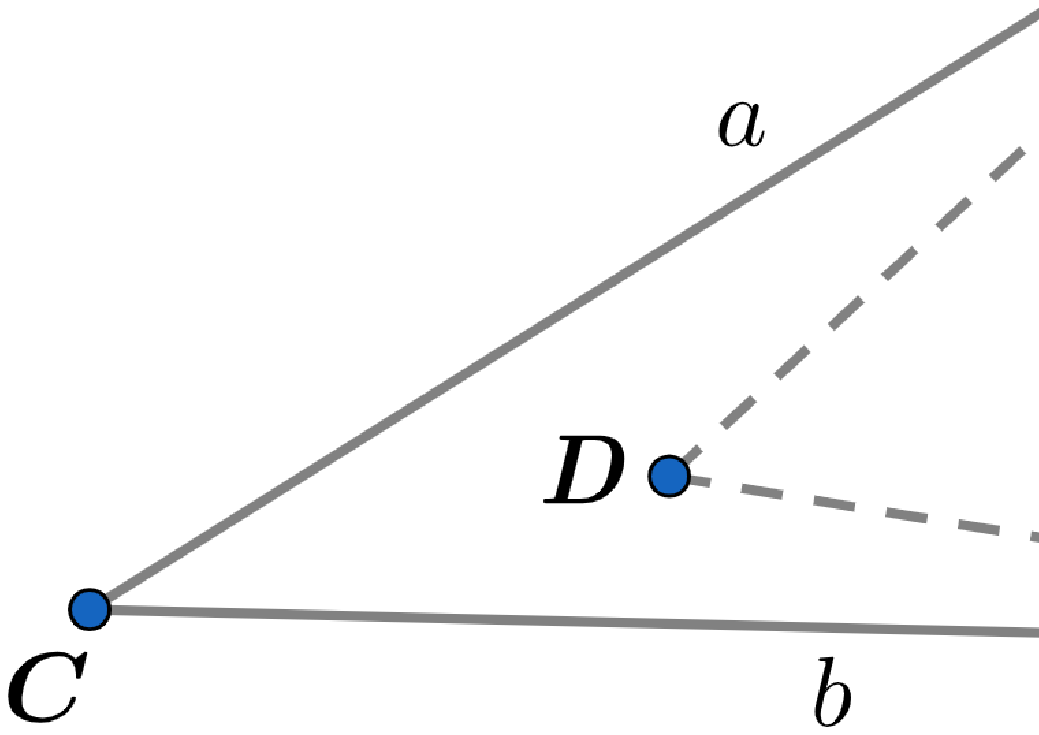} \\ a)}
\end{minipage}
\hspace{1mm}
\begin{minipage}[h]{0.48\textwidth}
\center{\includegraphics[width=0.9\textwidth, trim=0mm 0mm 0mm 0mm, clip]{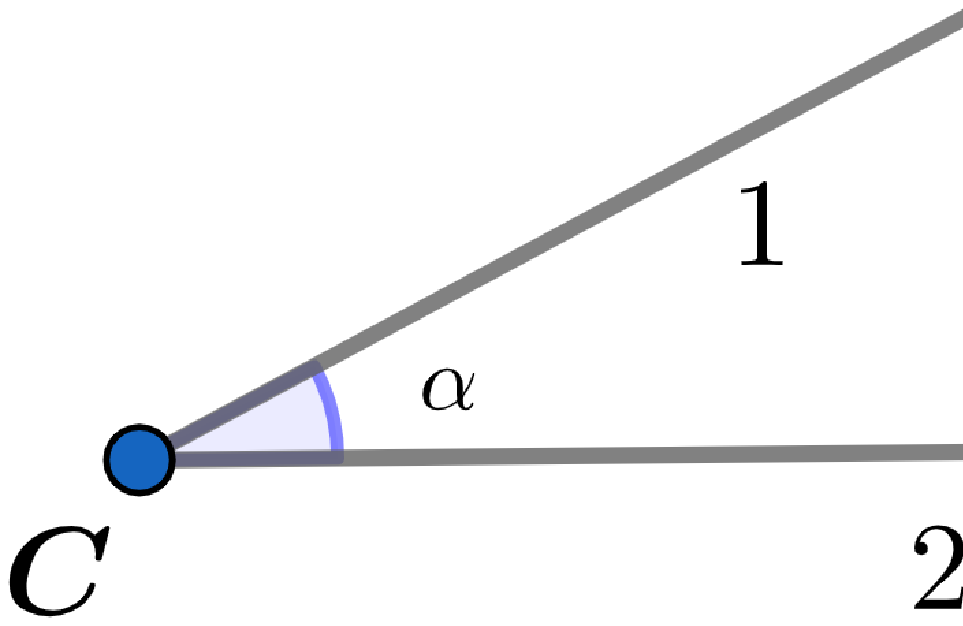} \\ b)}
\end{minipage}
\caption{
a) A triangle with the $\Delta$ property;
b) A ``narrow'' isosceles triangle with the $\Delta$ property.
}
\label{Fig_uniq}
\end{figure}

A regular triangle with unit side satisfies the $\Delta$ property and has perimeter $3$ (the smallest possible).
It should be noted that a triangle $ABC$ with the $\Delta$ property should not contain a regular triangle with unit side.
To show this let us consider a ``narrow'' isosceles triangle $ABC$ such that $\angle ACB =\alpha$, $|BA|=|CA|=2\cos \alpha$.
It is easy to see that this triangle satisfies the $\Delta$ property if $\alpha \leq \pi/3$, see Fig.~\ref{Fig_uniq}~b).
On the other hand, such triangle does not contain a regular triangle with unit side for sufficiently small $\alpha >0$.
\smallskip

Let us show that for any given $n\geq 4$, there are $n$-gons $P$ with the $\Delta$ property such that the perimeter $L(P)$ is arbitrarily close to $3$.
We construct such polygons in a small neighborhood (with respect to the Hausdorff metric) of a regular triangle with unit side.
Let us fix a small number $\varepsilon >0$. We start from the regular triangle $A_1A_2A_3$ with side length $1+\varepsilon$.
Further, we fix a number $\alpha>0$ such that $\alpha<\min\left\{\pi/n,\varepsilon\right\}$ and consider the points $A_4, A_5,\dots, A_{n-1}, A_n$ in the plane such that
$|A_1A_i|=1+\varepsilon$ and $\angle A_{i-1}A_1A_i=\alpha$, $i=4,5\dots,n$. Let $P$ be the convex hull of the points $A_i$, $i=1,2,\dots,n$, see Fig.~\ref{Fig_uniq2}.
For any side $A_iA_{i+1}$ of $P$ with $i=1,2,\dots,n-1$ we can easily choose a point $B\in P$ such that $|A_iB|=|A_{i+1}B|=1$.
The same is true for the side $A_1A_n$ for sufficiently small $\alpha>0$, i.~e., the polygon $P$ has the $\Delta$ property for sufficiently small $\alpha>0$.
For the perimeter of $P$ we have the following estimate:
\begin{eqnarray*}
L(P)&=&|A_1A_n|+|A_1A_2|+|A_2A_3|+\sum_{i=3}^{n-1} |A_iA_{i+1}|=3+(n-3)|A_3A_4|\\
&=&(1+\varepsilon)\cdot \bigl(3+2(n-3)\sin(\alpha/2)\bigr)<(1+\varepsilon)\cdot\bigl(3+(n-3)\alpha)\bigr).
\end{eqnarray*}
Now, it is clear that for sufficiently small $\varepsilon>0$ and $\alpha>0$, $L(P)$ is as close to $3$ as we want.

\begin{figure}[t]
\begin{minipage}[h]{0.45\textwidth}
\center{\includegraphics[width=0.9\textwidth, trim=0mm 0mm 0mm 0mm, clip]{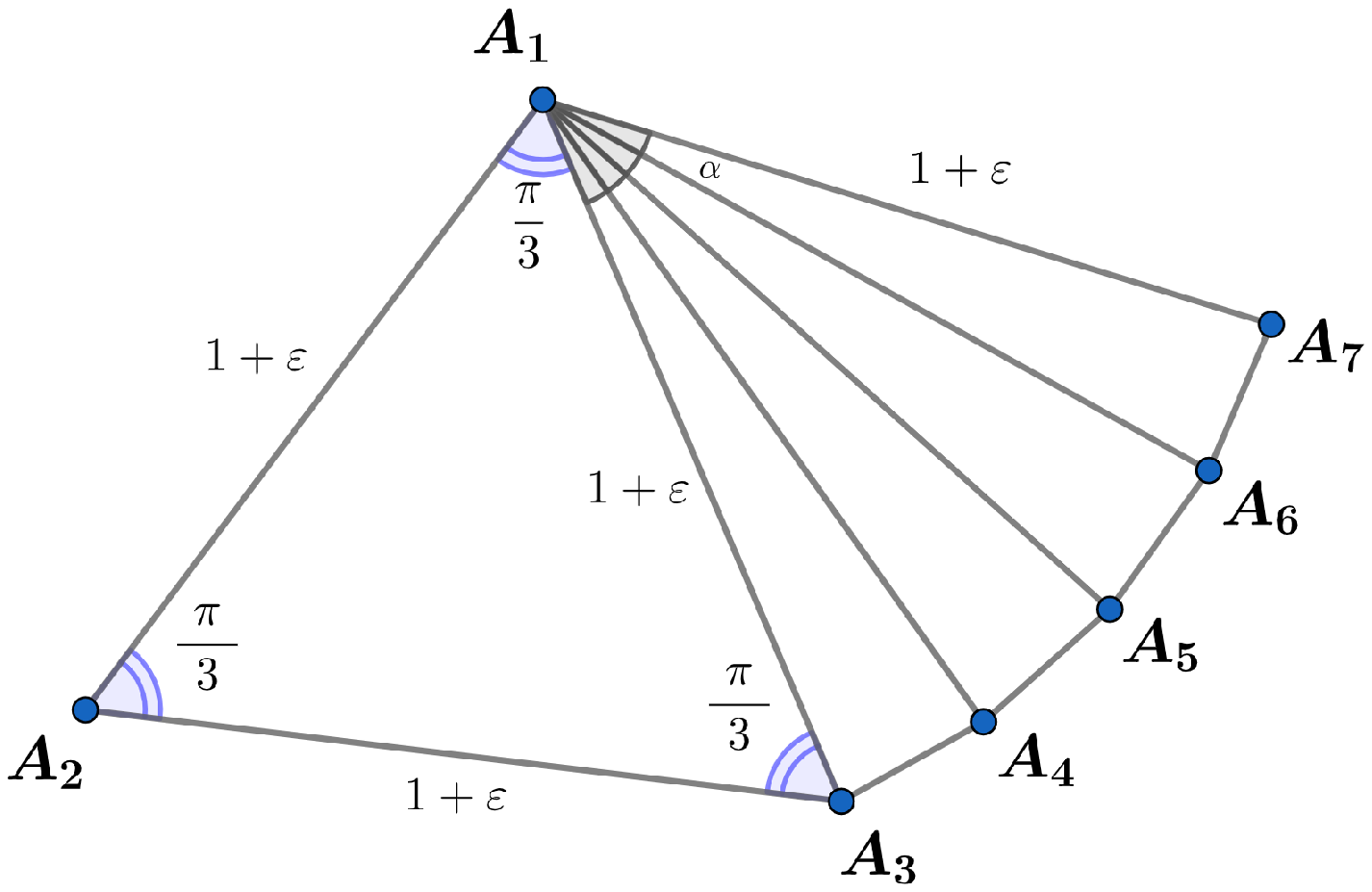} \\ a)}
\end{minipage}
\hspace{1mm}
\begin{minipage}[h]{0.45\textwidth}
\center{\includegraphics[width=0.9\textwidth, trim=0mm 0mm 0mm 0mm, clip]{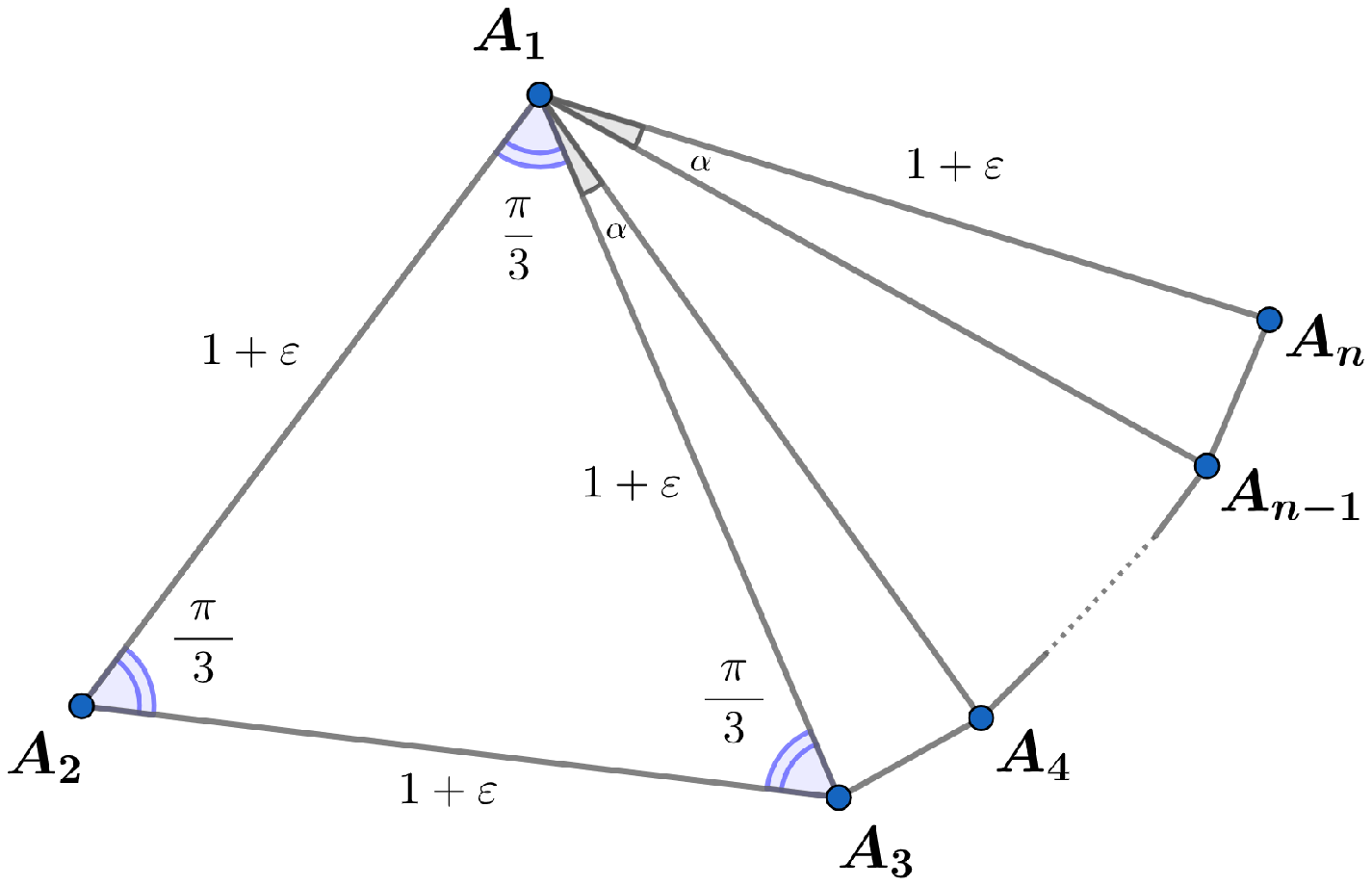} \\ b)}
\end{minipage}
\caption{ An $n$-gon with the $\Delta$ property for
a) $n=7$;
b) $n\geq 4$.
}
\label{Fig_uniq2}
\end{figure}

It is reasonable to try to weaken the conditions of Theorem \ref{th.minim}.
The most natural variant is to prescribe the presence of struts not on all sides of the polygon, but {\it only on some of them}.

The presence of a strut for only one side of the polygon $P$ is not enough for the inequality $L(P)\geq 3$.
Indeed, we can consider triangles $ABC$ with $|AC|=|BC|=1$ and with very short $AB$.
In this case the side $AB$ has a strut, but $L(P)$ can be smaller than $2+\varepsilon$ for any given $\epsilon >0$.
On the other hand, if we assume that a side $AB$ has a strut and $|AB|\geq 1$, then obviously we get $L(P)\geq 3$.

It is interesting that a similar idea works in the case of two  sides of $P$, adjacent to each other, but
the corresponding result is much more difficult to obtain.
Our second main result is the following.

\begin{theorem}\label{th.minim.2side}
Given a natural number $n \geq 3$. Let $P$ be a convex polygon with the consecutive vertices $A_1,A_2,A_3,\dots, A_{n-1},A_n$
such that the sides $A_1A_2$ and $A_2A_3$ have struts and $|A_1A_2|+|A_2A_3|\geq 1$. Then the perimeter $L(P)$ of $P$ satisfies the inequality
$L(P) \geq 3$.
\end{theorem}

It should be noted that a regular triangle with unit side is not a unique polygon with $L(P)=3$ in the above theorem. There are a continuous family of quadrangles with
this property and one distinguished pentagon, see the details in Propositions \ref{pr.perpen3p} and \ref{pr.perpen3}.
\smallskip

The above examples of polygons with the perimeters close to $3$ and Theorem \ref{th.minim.2side} naturally  entail the following

\begin{conjecture}
Given natural numbers $n \geq 3$ and $m=1,\dots,n-1$. Let $P$ be a convex polygon $P$ with the consecutive vertices $A_1,A_2,A_3,\dots, A_{n-1},A_n$
such that the sides $A_1A_2, A_2A_3, \dots, A_mA_{m+1}$ have struts and $|A_1A_2|+|A_2A_3|+\cdots +|A_mA_{m+1}|\geq 1$.
Then the perimeter $L(P)$ of $P$ satisfies the inequality
$L(P) \geq 3$.
\end{conjecture}

The paper is organized as follows. In Section \ref{sect.1} we consider some  important information on convex figures and
prove some useful results related to the difference bodies of convex polygons (in particular, with the $\Delta$ property).
Moreover, some results (in particular, Theorem~\ref{th.centsym}) on centrally symmetric polygons are obtained.
In Section \ref{sect.2} we study some special pentagons and prove Theorem \ref{th.perpent}.
In the final Section \ref{sect.3} we prove the main results of the paper.
\smallskip

The authors are grateful for Endre Makai, Jr. and Horst Martini for useful discussions of issues related to the subject of this paper.

\section{Notation and some auxiliary results}\label{sect.1}

The following property (that of the monotonicity of the perimeter) of convex figures is well-known (see, e.g., \cite[\S 7]{BoFe1987}).

\begin{prop}\label{monotper}
If convex figures $K_1$ and $K_2$ in the Euclidean plane are such that $K_1\subset K_2$,
then their perimeters satisfy the inequality $L(K_1) \leq L(K_2)$, and the equality holds if and only if $K_1=K_2$.
\end{prop}

\smallskip

If $X$ and $Y$ are sets in $\mathbb{R}^2$, then $X+Y=\left\{x+y\,|\, x\in X, y\in Y \right\}$
denotes {\it the Minkowski sum} of $X$ and $Y$. For given $t \in\mathbb{R}$ and $X\subset \mathbb{R}^2$, we consider $tX:=\left\{t\cdot x\,|\, x\in X \right\}$.
In what follows, we will often use $-X=\left\{-x\,|\, x\in X \right\}$.
Important properties and numerous applications of these operations can be found, for instance, in \cite{BoFe1987, Gard2006, MaMoOl}.

We will need a special transformation of convex polygons in $\mathbb{R}^2$.
For a convex polygon~$P$, we define {\it the difference body} by the formula $D(P)= P+(-P)$.
Note that {\it the central symmetral} of $P$, defined by $\lozenge(P)=\frac{1}{2} P+\frac{1}{2}(-P)$, differs from  the difference body $D(P)$
only by a dilatation factor of $1/2$, see e.~g.
\cite[P.~106]{Gard2006} or \cite{BiGaGro2017}.
It should be noted that the central symmetral has many important properties.
For instance, it is easy to check that the perimeters of the polygons $P$ and  $\lozenge(P)$ coincide (hence, the perimeter of $D(P)$ is two times larger).

\begin{remark}
The central symmetral was an important element
of the solution of one problem by L. Fejes T\'{o}th in \cite{NikRas2002}.
Let $P$ be a plane convex $n$-gon with side lengths $a_1,\dots,a_n$. Let $b_i$ be the length of the longest chord of $P$ parallel to the $i$-th side.
Then $3\leq\sum_{i=1}^n a_i/b_i \leq 4$, whereas $\sum_{i=1}^n a_i/b_i=3$ if and only if $P$ is a snub triangle obtained by cutting off
three congruent triangles from the corners of a triangle, while $\sum_{i=1}^n a_i/b_i=4$ if and only if $P$ is a parallelogram,
as it was conjectured and partially proven by L. Fejes T\'{o}th in \cite{FT1970}.
\end{remark}

It is clear that the difference body $D(P)$ of any convex polygon $P$ is a centrally symmetric polygon. In particular, the origin $O$ is the center of $D(P)$.
It is remarkable that the difference body of a polygon $P$ with the $\Delta$ property also has one special property.

We will say that a centrally symmetric polygon $P$ has {\it the $\Delta^s$ property} if for any its side $AB$, $P$ contains a rectangle $KLMN$ such that
$\overrightarrow{KL}=\overrightarrow{NM}=\overrightarrow{AB}$ and the distance from $O$ to any vertex of $KLMN$ is $1$.

The following result is very important for our goals.

\begin{prop}\label{pr.syssym}
Let $P$ be a polygon with the $\Delta$ property. Then
its difference body $D(P)=P+(-P)$ has the $\Delta^s$ property. Moreover, the perimeter $L\bigl(D(P)\bigr)$ is equal to $2\cdot L(P)$.
\end{prop}

\begin{proof}
We may assume that $P$ does not have two parallel sides (the general case could be obtained by passing to the limit).
Let us fix a side $AB$ of $P$. By the $\Delta$ property, there is a point $C\in P$ such that $|AC|=|BC|=1$.
Since $\triangle ABC \subset P$, then $D(\triangle ABC)\subset D(P)$.
On the other hand, it is clear that $D(\triangle ABC)=\triangle ABC +(-\triangle ABC)$ is a centrally symmetric convex hexagon,
which also has a symmetry axis orthogonal to
$AB$. Without loss of generality we may even assume that $C=O$ (if we move the origin, then we get a parallel translation of $D(P)$).
In this case $D(\triangle ABC)$ contains both $\triangle ABC$ and $-\triangle ABC$,
hence the rectangle $AB(-A)(-B)$, and that proves the first assertion.

To prove the second assertion, we note that any side of $D(P)$ is  parallel to some side either of $P$ or of $-P$.
Moreover, for any side $AB$ of $P$ there are exactly two sides of $D(P)$ that are parallel to $AB$ and which have the same length
(here we have used that $P$ has no parallel sides). Hence, $L(D(P))=2\cdot L(D)$. The proposition is proved.
\end{proof}
\smallskip

\begin{figure}[t]
\center{\includegraphics[width=0.7\textwidth, trim=0mm 0mm 0mm 1mm, clip]{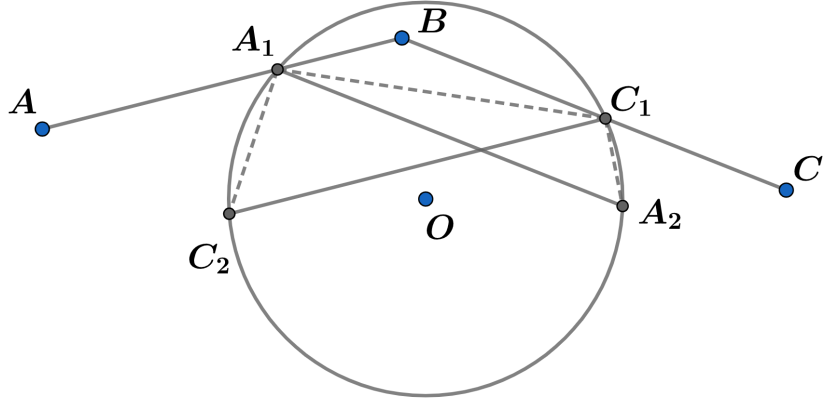} }
\caption{
The key idea in the proof of Theorem \ref{th.centsym}
}
\label{Fig_uniq4}
\end{figure}

The following theorem is very important for the proof of Theorem \ref{th.minim}.

\begin{theorem}\label{th.centsym}
Let $P$ be a centrally symmetric convex polygon with the $\Delta^s$ property and such that any side of it is not longer than $1$. Then
the perimeter $L(P)$ is not less than~$6$.
Moreover, $L(P)=6$ if and only if $P$ is a regular hexagon with side of length $1$.
\end{theorem}

\begin{proof}
We denote the the circle $S(O,1)$ by $S$ for brevity. Suppose that a side $AB$ of $P$ is in the unit ball $B(O,1)$. Then, by the $\Delta^s$ property,
both vertices $A$ and $B$ are in $S$.
Indeed, if at least one of the points $A$ and $B$ is in the interior of $B(O,1)$, then the intersection of the straight line $AB$ with $B(O,1)$
has length larger than $|AB|$.
By the $\Delta^s$ property, $P$ contains a rectangle $KLMN$ such that
$\overrightarrow{KL}=\overrightarrow{NM}=\overrightarrow{AB}$ and the distance from $O$ to any vertex of $KLMN$ is $1$.
Hence, $K,L,M,N \in S \cap P$. Therefore, the line $AB$ should be somewhere strictly between the straight lines $KL$ and $NM$.
On the other hand, the straight line $AB$ is a supporting line for $P$, what is impossible.

Now, let us suppose that a vertex $B$ is in the interior of $B(O,1)$. Then the vertices $A$ and $C$ of $P$, that are adjacent to $B$, could not lie in
$B(O,1)$ by the same arguments as above. Denote by $A_1$ and $C_1$ the intersections of $S$ with the segments $[A,B]$ and $[B,C]$, respectively.
Let us prove that $|A_1C_1|< 1$.
Let us consider points $A_2,C_2 \in S$ such that the straight line $A_1A_2$ is parallel to $BC$ and the straight line $C_1C_2$ is parallel to $BA$,
see Fig.~\ref{Fig_uniq4}. By the $\Delta^s$ property, $|A_1A_2|\leq |BC|\leq 1$ as well as  $|C_1C_2|\leq |AB|\leq 1$
(if, say, $|C_1C_2|>|AB|$, then it is impossible to choose points $K,L,M,N \in S\cap P$ such that $\overrightarrow{KL}=\overrightarrow{NM}=\overrightarrow{AB}$).
Since $\angle A_1C_1 A_2 >\pi/2$, then $|A_1C_1| < |A_1A_2|\leq 1$.

Now, let us consider $\mathcal{M}$, the intersection of $S$ with the boundary $\operatorname{bd} (P)$. Is is clear that $\mathcal{M}$ is a finite set.
Note that $\mathcal{M}$ could be empty. In this case $S$ is a subset of $P$ by the above arguments.
Then $L(P) \geq L(B(O,1))=2\pi >6$ by Proposition \ref{monotper}. In what follows we assume that $\mathcal{M}\neq \emptyset$.

Let $m$ be the cardinality of $\mathcal{M}$. Then we enumerate element of $\mathcal{M}$ by $M_i$, $i=1,2,\dots,m$ in such a way that
the parts of $\operatorname{bd} (P)$ between $M_{l}$ and $M_{l+1}$, $l=1,\dots,m$ (we assume that $M_{m+1}:=M_1$), have no common point with $S$.

Let us consider $I=\{i \in \mathbb{N}\,|\, 1\leq i \leq m\}$ and $I_1, I_2\subset I$ which are defined as follows: $i \in I_1$ ($i \in I_2$) if and only if
the part of $\operatorname{bd} (P)$ between $M_{i}$ and $M_{i+1}$ is situated in $B(O,1)$ (respectively, is situated outside of $B(O,1)$).
It is clear that $I_1 \cup I_2=I$.

Now, denote by $l_i$, $i=1,2,\dots,m$, the length of the part of $\operatorname{bd} (P)$ between the points $M_i$ and $M_{i+1}$.
Further, denote by $\lambda_i$, $i=1,2,\dots,m$, the length of the (smaller) arc of $S$ between the points $M_i$ and $M_{i+1}$.

For all $i\in I_1$, from the above reasoning, we know that the part of $\operatorname{bd} (P)$ between the points $M_i$ and $M_{i+1}$
is either a part of one side of $P$ or the union of parts of two adjacent to each other sides of $P$. Moreover, $|M_iM_{i+1}|\leq 1$ and
$|M_iM_{i+1}|=1$ only if both $M_i$ and $M_{i+1}$ are the endpoints of a side of $P$ with length $1$. It is easy to see that
$l_i \leq |M_iM_{i+1}| <\lambda_i$, but $l_i\geq \frac{3}{\pi} \,\lambda_i$ due to $|M_iM_{i+1}|\leq 1$.

For any $i\in I_2$, the convex hull of the part of $\operatorname{bd} (P)$ between the points $M_i$ and $M_{i+1}$
contains the arc of $S$ between the points $M_i$ and $M_{i+1}$. By Proposition \ref{monotper}, we get
$|M_iM_{i+1}|+l_i \geq |M_iM_{i+1}|+\lambda_i$, implying $l_i \geq \lambda_i$ (it is clear that equality is impossible in this case).

Therefore, we have
\begin{eqnarray*}
L(P)=\sum_{i\in I_1} l_i+ \sum_{i\in I_2} l_i \geq \frac{3}{\pi} \sum_{i\in I_1} \lambda_i+
\sum_{i\in I_2} \lambda_i
\geq  \frac{3}{\pi} \sum_{i\in I} \lambda_i= \frac{3}{\pi} \cdot 2\pi =6.
\end{eqnarray*}
If $L(P)=6$, then $I_2=\emptyset$ and every line segment $[M_i,M_{i+1}]$ is a side of $P$ with length~$1$.
Therefore, $P$ is a regular hexagon with side of length $1$.
\end{proof}

It should be noted that a small modification of the arguments used in the proof of Theorem \ref{th.centsym} imply the following result.

\begin{prop}\label{pr.centsym}
If a sequence of a centrally-symmetric convex polygon  $\left\{P_{n}\right\}$, $n\in \mathbb{N}$, with the $\Delta^s$ property and with all sides not longer than $1$,
is such that $L(P_n) \rightarrow 6$ as $n \rightarrow \infty$, then there is a subsequence of this sequence that converges to
a regular hexagon with side of length $1$ in the Hausdorff metric.
\end{prop}

\begin{proof}
Since the space of all convex closed subsets of any ball $B(x,r) \subset \mathbb{R}^2$ is compact with respect to the Hausdorff metric,
me may assume that $P_n \rightarrow Q$ as $n \rightarrow \infty$ in this metric (see e.~g. \cite[4.3.2]{Had1957}),
where $Q$ is some centrally symmetric figure in $\mathbb{R}^2$.
Let us supply all characteristics of $P$ in the proof of Theorem \ref{th.centsym} by the index $n$ in order to consider the corresponding characteristics of
the polygon $P_n$.
We know that
\begin{eqnarray*}
L(P_n)=\sum_{i\in I_1^n} l_i^n+ \sum_{i\in I_2^n} l_i^n \geq \frac{3}{\pi} \sum_{i\in I_1^n} \lambda_i^n+
\sum_{i\in I_2^n} \lambda_i^n
\geq  \frac{3}{\pi} \sum_{i\in I^n} \lambda_i^n+\kappa_n
=6+\kappa_n,
\end{eqnarray*}
where $\kappa_n=\left(1- \frac{3}{\pi}\right)\sum_{i\in I_2^n} \lambda_i^n$.
Since $L(P_n) \rightarrow 6$ as $n \rightarrow \infty$, then $\sum_{i\in I_2^n} \lambda_i^n \rightarrow 0$.
This means that  we get $Q \subset B(O,1)$ after passing to the limit.

For any $P_n$, the distance between every adjacent points of intersection of the boundary $\bd (P_n)$ with the circle $S=S(O,1)$ does not exceed $1$.
Hence, the boundary of $Q$ has the same property. Let us consider $\mathcal{L}=  \bd (Q) \cap S$.
Since $\co (\mathcal{L}) \subset Q$, then we have the inequality $L(Q)\geq L(\operatorname{co} (\mathcal{L}))$ for their perimeters.

If the cardinality of $\mathcal{L}$ is $\geq 7$,
then we can choose a finite sequence $C_1,\dots, C_k, C_{k+1}=C_1$ of points in $\mathcal{L}$ such that $|C_iC_{i+1}|\leq 1$, $i=1,\dots,k$, $k\geq 7$.
It is easy to see that $L(Q)\geq \sum_{i=1}^{k} |C_iC_{i+1}| >6$, which is impossible. If the cardinality of  $\mathcal{L}$ is $\leq 6$,
then it is exactly $6$ and the points of $\mathcal{L}$ are the vertices of a regular hexagon with unit side. Therefore, $L(Q)\geq 6$ with $L(Q)=6$
if $Q$ is the convex hull of $\mathcal{L}$.
\end{proof}

\begin{prop}\label{pr.proreghex}
Suppose that a polygon $P$ is such that
its difference body $D(P)=P+(-P)$ is a regular hexagon with unit side.
Then $P$ is obtained from a regular triangle with side of length $1+a$, $a\in [0,1/2]$, by cutting off
three congruent triangles with sides of length $a$ from its corners. If $a\neq 0$, then such $P$ has not the $\Delta$ property.
\end{prop}

\begin{proof}
Let us consider the minimal (say, by area) regular triangle $T$ with sides parallel to sides of $D(P)$ that contains the polygon $P$.
Let $1+a$ be the length of any side of $T$ (it is clear that $a\geq 0$).
Since every side of $P$ is parallel to some side of $D(P)$, we see that on any side of $T$, there is a side of $P$.
Since the distance between straight lines through two parallel sides of $D(P)$ is $\sqrt{3}$, then the distance between straight lines through two parallel sides of $P$
is $\sqrt{3}/2$. Hence, any straight line through a side of $P$, that is not on a part of a side of $T$, should cut a regular triangle with sides of length $a$
from a corner of $T$.
Hence, $P$ has three pairs of parallel sides with lengths $a$ and $1-a$. By construction of $T$ we have $a\leq 1-a$ (otherwise we can find a triangle $T'$
with sides parallel to sides of $D(P)$ that is smaller than $T$). Therefore, $a\in [0,1/2]$. For $a=0$ we obtain a regular triangle with unit side.
On the other hand, $D(P)$ is a hexagon for $a\in (0,1/2]$.

Finally, it is easy to see that $P$ has not the $\Delta$ property for $a\in (0,1/2]$. Indeed, if $AB$ is a side of $P$ with $|AB|=a$, then
there is no $C\in P$, such that $|AC|=|BC|=1$. Otherwise the distance between the line $AB$ and the supporting line to $P$, which is parallel to $AB$,
is at least $\sqrt{1-(a/2)^2}$. On the other hand, this distance is $\sqrt{3}/2$, hence, $\sqrt{3}/2 \geq \sqrt{1-(a/2)^2}$, that implies $a\geq 1$.
This contradiction proves the second assertion of the proposition.
\end{proof}

\section{One problem on special pentagons}\label{sect.2}

In this section we consider some properties of special convex pentagons in the Euclidean plane.
The main result of this section is the following.

\begin{theorem}\label{th.perpent}
Let $P$ be a convex pentagon in the Euclidean plane with the consecutive vertices $A, B, C, E, F$ such that $|AE|=|BE|=|BF|=|CF|=1$ and $|AB|+|BC|\geq 1$.
Then the perimeter $L(P)$ of this pentagon is not less than $3$.
\end{theorem}

We allow the pentagon $P$ to degenerate into
a polygon with fewer sides.
For example, $B$ may degenerately lie in the segment $[A,C]$, $E$ may coincide with $C$ or $F$, $F$ may coincide with $E$ or $A$.
This assumption allows us to consider a compact subset of polygons in the plane with respect to the Hausdorff metric.
In particular, one can be sure that there exists a pentagon with a minimum perimeter.
\smallskip

Let us consider the following values: $\alpha=\angle EAB=\angle EBA$, $\beta=\angle FBC= \angle FCB$, $\gamma= \angle BEF =\angle BFE$,
$\theta=\angle EBF =\pi-2\gamma$.
It is clear that $|AB|=2\cos \alpha$, $|BC|=2\cos \beta$, $|EF|=2 \cos \gamma =2 \sin (\theta/2)$, see Fig. \ref{Fig_arch}~a).

\begin{remark}\label{re.eqv.pent}
In terms of the angles $\alpha, \beta, \gamma$, one can give a complete description of the cases when the equality $L(P)=3$  holds in Theorem \ref{th.perpent}.
From Propositions \ref{pr.perpen2}, \ref{pr.perpen3p}, and~\ref{pr.perpen3} we see that $L(P)=3$ exactly for the following values
of $(\alpha, \beta, \gamma)$: $(\pi/3,\pi/3,\pi/3)$, which corresponds to  $E=C$, $F=A$, $|EF|=|AB|=|BC|=1$;
$(\pi/3,\pi/2,\pi/3)$, which corresponds to  $B=C$, $A=F$, $|EF|=|AB|=|BE|=1$; $(\pi/2,\pi/3,\pi/3)$, which corresponds to  $B=A$, $C=E$, $|BF|=|BC|=|EE|=1$;
$(\alpha,\beta,\pi/2)$ with $\cos(\alpha)+\cos(\beta)=1$, which corresponds to a family of quadrangles with $E=F$ and $|AB|+|BC|=1$;
$(\arccos(1/4),\arccos(1/4),\arccos(1/4))$,
which correspond to a special pentagon with perimeter $3$, see Fig. \ref{Fig_arch}~b).
\end{remark}

\begin{remark}\label{re.eqv.pent2}
The pentagon in Fig. \ref{Fig_arch}~b) with $\alpha=\beta=\gamma=\arccos(1/4)$ is quite remarkable. Note that
$\angle CEB= \angle CFB =\angle CAB = \angle AFB =\angle AEB=\angle ACB= \angle ECF=  \angle EBF = \angle EAF=2\arcsin(1/4)=\arccos(7/8)$.
In particular, all vertices of this pentagon lie on the same circle. Moreover, $|AC|=7/8$.
Thus, if we increase this pentagon by $8$ times, we get an {\it integer pentagon}, in which all sides and diagonals have integer lengths.
This pentagon was found at first in \cite{Mu1953}, see also the discussion in \cite[P.~19]{PeJo1997}.
\end{remark}

\begin{figure}[t]
\begin{minipage}[h]{0.45\textwidth}
\center{\includegraphics[width=0.98\textwidth, trim=0in 0in 0mm 0mm, clip]{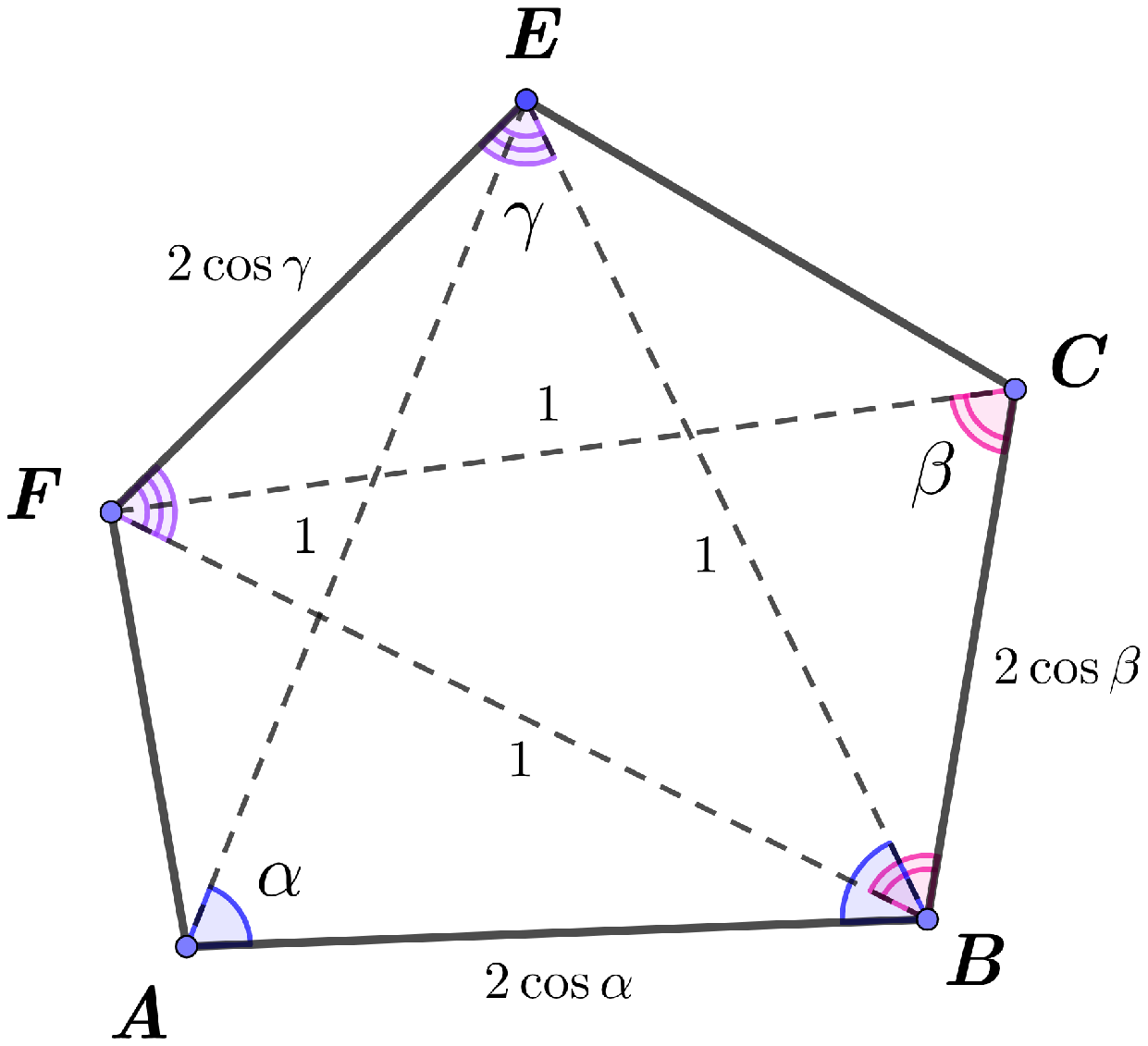} \\ a)}
\end{minipage}
\hspace{3mm}
\begin{minipage}[h]{0.45\textwidth}
\center{\includegraphics[width=0.94\textwidth, trim=0mm 0in -2mm 0mm, clip]{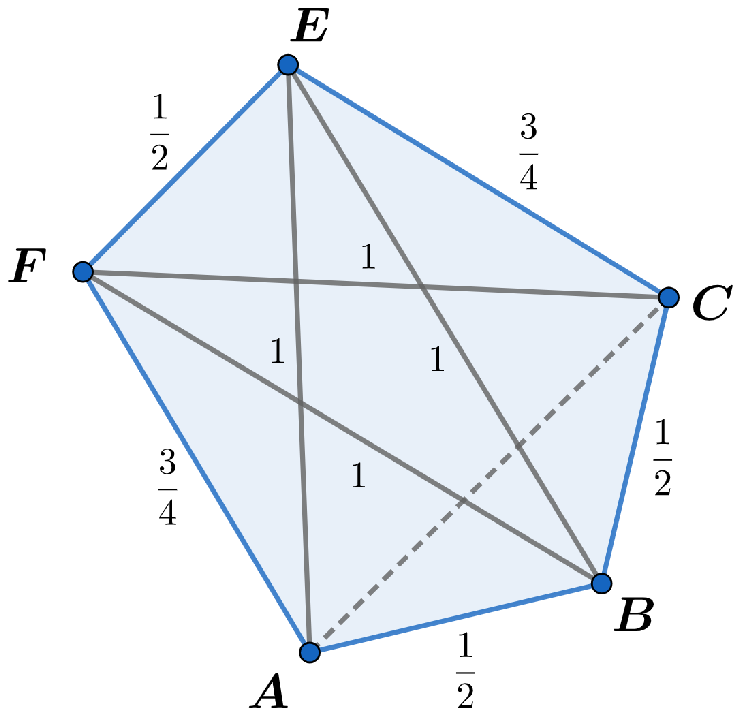} \\ b)}
\end{minipage}
\caption{
a) A pentagon $P$ as in Theorem \ref{th.perpent};
b) A special pentagon $P$ with perimeter $3$.
}
\label{Fig_arch}
\end{figure}

For given $\theta, \gamma \in [0,\pi/2]$ with $2\gamma+\theta=\pi$, let us consider the function
\begin{eqnarray}\label{eq.f1}
f(x)&=&1+4 \cos^2(x)-4\cos(x)\cos(x-\theta) \notag\\
&=&3+2\cos(2x)-2\cos(2x-\theta)-2\cos(\theta)\notag\\
&=&3+2\cos(2x)+\cos(2\gamma)+\cos(2x+2\gamma)\notag\\
&=&1+4\cos^2(x)+4\cos(x)\cos(x+2\gamma)\,.
\end{eqnarray}

Obviously, $f^{\prime}(x)=4\left(\sin(2x-\theta)-\sin(2x)\right)$.
If $\theta\in (0,\pi/3]$ and  $x \in [\pi/3,\pi/2]$, then  $\sin(2x-\theta)>\sin(2x)$, and $f(x)$ increases for $x\in [\pi/3,\pi/2]$.
In particular, $f(\pi/3)=2-2\cos(\pi/3-\theta)<2-2\cos(\pi/3)=1=f(\pi/2)$.

It is easy to check that $|AF|=\sqrt{f(\alpha)}$ and $|CE|=\sqrt{f(\beta)}$. Therefore, we have the following explicit expressions for the perimeter $L(P)$:

\begin{equation}\label{eq.per1}
L(P)=\mathcal{A}(\alpha,\beta,\gamma)=\mathcal{B}(\alpha,\beta,\theta)=\mathcal{C}(u,v,\gamma),
\end{equation}
where

\begin{eqnarray}\label{eq.per2}
\mathcal{A}(\alpha,\beta,\gamma)&=&2 \bigl(\cos(\alpha)+ \cos (\beta)+ \cos(\gamma)\Bigr)+\sqrt{1+4\cos^2(\alpha)+4\cos(\alpha)\cos(\alpha+2\gamma)} \notag\\
&&+\sqrt{1+4\cos^2(\beta)+4\cos(\beta)\cos(\beta+2\gamma)},
\end{eqnarray}

\begin{eqnarray}\label{eq.per3}
\mathcal{B}(\alpha,\beta,\theta)&=&2 \bigl(\cos(\alpha)+ \cos (\beta)+ \sin(\theta/2)\Bigr)+\sqrt{1+4\cos^2(\alpha)-4\cos(\alpha)\cos(\alpha-\theta)} \notag\\
&&+\sqrt{1+4\cos^2(\beta)-4\cos(\beta)\cos(\beta-\theta)},
\end{eqnarray}

\begin{eqnarray}\label{eq.per4}
\mathcal{C}(u,v,\gamma) = 2u+2v+2\cos(\gamma)+\sqrt{1+8u^2\cos^2(\gamma)-8u\sqrt{1-u^2}\sin(\gamma)\cos(\gamma)} \notag\\
+\sqrt{1+8v^2\cos^2(\gamma)-8v\sqrt{1-v^2}\sin(\gamma)\cos(\gamma)},
\end{eqnarray}
with $u=\cos(\alpha)$ and $v=\cos(\beta)$.
We will use all these formulas for different purposes.

\smallskip

\begin{remark}\label{re.equal0}
If $\alpha \leq \pi/3$, then $\angle AEB \geq \pi/3$ and the perimeter of the triangle $\triangle ABE$ is $\geq 3$. By
Proposition \ref{monotper} we have $L(P)\geq 3$ with equality only if $F=A$ and $C=E$.
The same we have under the assumption $\beta \leq \pi/3$ ($\angle BFC \geq \pi/3$, and the perimeter of the triangle $\triangle BFC$ is $\geq 3$).
If $\gamma \leq \pi/3$ (this is equivalent to $\theta \geq \pi/3$), then $\angle EBF \geq \pi/3$ and the perimeter of the triangle $\triangle EBF$ is $\geq 3$.
In this case we have also $L(P)\geq 3$ with equality only if $F=A$ and $C=E$.

Note also that for $\gamma=\pi/2$, we get $E=F$ in the pentagon $P$, $|AE|=|CF|=1$ and $|AB|+|BC|\geq 1$. Therefore, $L(P)\geq 3$ with
$L(P)=3$ only when $|AB|+|BC|=1$.
\end{remark}

{\bf In what follows within this section, we assume that $\alpha, \beta \in [\pi/3,\pi/2]$ and  $\gamma \in [\pi/3,\pi/2)$ ($0<\theta \leq \pi/3)$.}
\smallskip

Our strategy for proving Theorem \ref{th.perpent} is as follows.
Let us suppose that there is a pentagon $P$ that satisfies all assumptions of this theorem and $L(P)<3$.
Then among all pentagons (possibly, degenerate) satisfying the assumptions of Theorem \ref{th.perpent}, there is pentagon $P_0$
which reaches the absolute minimum of the perimeter with a value less than $3$.
Note that $P_0$ could not be a degenerate pentagon.

Let us consider the set
\begin{equation}\label{eq.setper}
\Omega= \left\{(\alpha,\beta,\gamma) \in [\pi/3,\pi/2]^3 \subset \mathbb{R}^3 \,|\, \cos(\alpha)+\cos(\beta)\geq 1/2 \right\}.
\end{equation}
We will find the point of absolute minimum for the function $L(P)$  on this set, see (\ref{eq.per1}).
\smallskip

{\bf The first step.}
Let us describe the point of $\Omega$, where $\frac{\partial \mathcal{B}}{\partial \alpha} (\alpha,\beta,\theta)=0$.
By direct computations we get that
$$
\frac{\partial \mathcal{B}}{\partial \alpha} (\alpha,\beta,\theta)\cdot \sqrt{3+2\cos(2\alpha)-2\cos(\theta)-2\cos(2\alpha-\theta)}=2 \mathcal{D}(\alpha,\theta),
$$
where
$$
\mathcal{D}(\alpha,\theta)=\sin(2\alpha-\theta)-\sin(2\alpha)-\sin(\alpha)\sqrt{3+2\cos(2\alpha)-2\cos(\theta)-2\cos(2\alpha-\theta)}.
$$
Therefore, we should solve the equation $\mathcal{D}(\alpha,\theta)=0$
(note that this equation does not depend on $\beta$). Fortunately, this can be done explicitly.

Note that $\mathcal{D}(\alpha,0)=-2\sin(\alpha)<0$ and $\mathcal{D}(0,\pi/3)=\mathcal{D}(5\pi/12,\pi/3)=0$, whereas
$\mathcal{D}(\alpha,\pi/3)> (<)\,0$ for $\alpha \in (0,5\pi/12)$ (respectively, for $\alpha \in (5\pi/12,\pi/2]$).

It could be also checked that
$\frac{\partial^2 \mathcal{B}}{\partial \alpha \partial \theta} (\alpha,\beta,\theta)>0$ for $\alpha \in [\pi/3,\pi/2]$ and $\theta\in (0,\pi/3)$.
This means that for any given $\alpha \in [\pi/3,\pi/2]$ there is at most one $\theta \in [0,\pi/3]$ such that $\mathcal{D}(\alpha,\theta)=0$.
Taking into account the values of $\mathcal{D}(\alpha,0)$ and $\mathcal{D}(\alpha,\pi/3)$ for $\alpha \in [\pi/3,\pi/2]$, we can say more:
There is exactly one such $\theta$ for all $\alpha \in [0,5\pi/12]$ and there is no such $\theta$ for all $\alpha \in (5\pi/12,\pi/2]$.

We can produce the above values of $\theta$ explicitly. To solve the equation $\mathcal{D}(\alpha,\theta)=0$, one can use the substitution
$w=\tan(\theta)$. Since we consider only $(\alpha,\theta) \in [\pi/3,\pi/2]\times [0,\pi/3]$,
we obtain that $\mathcal{D}(\alpha,\theta)=0$ if and only if $\theta=g(\alpha)$, where $\alpha \in [0,5\pi/12]$ and
\begin{equation}\label{eq.perda1}
g(\alpha)=
\alpha\! - \!\arctan\!\! \left( \frac{\sin(2\alpha)\!\left(2\cos(\alpha)\cos(2\alpha) +\sqrt{16\cos^6(\alpha)-16\cos^4(\alpha)+1}\right)}
{8\cos^3(\alpha)\sin^2(\alpha) - \cos(2\alpha)\sqrt{16\cos^6(\alpha)-16\cos^4(\alpha)+1}}\right).
\end{equation}

The function $g(\alpha)$ is continuous and convex on the interval $[0,5\pi/12]$, whereas $g(0)=g(5\pi/12)=\pi/3$, see Fig. \ref{Fig_difa}.

\begin{figure}[t]
\begin{minipage}[h]{0.4\textwidth}
\center{\includegraphics[width=0.93\textwidth]{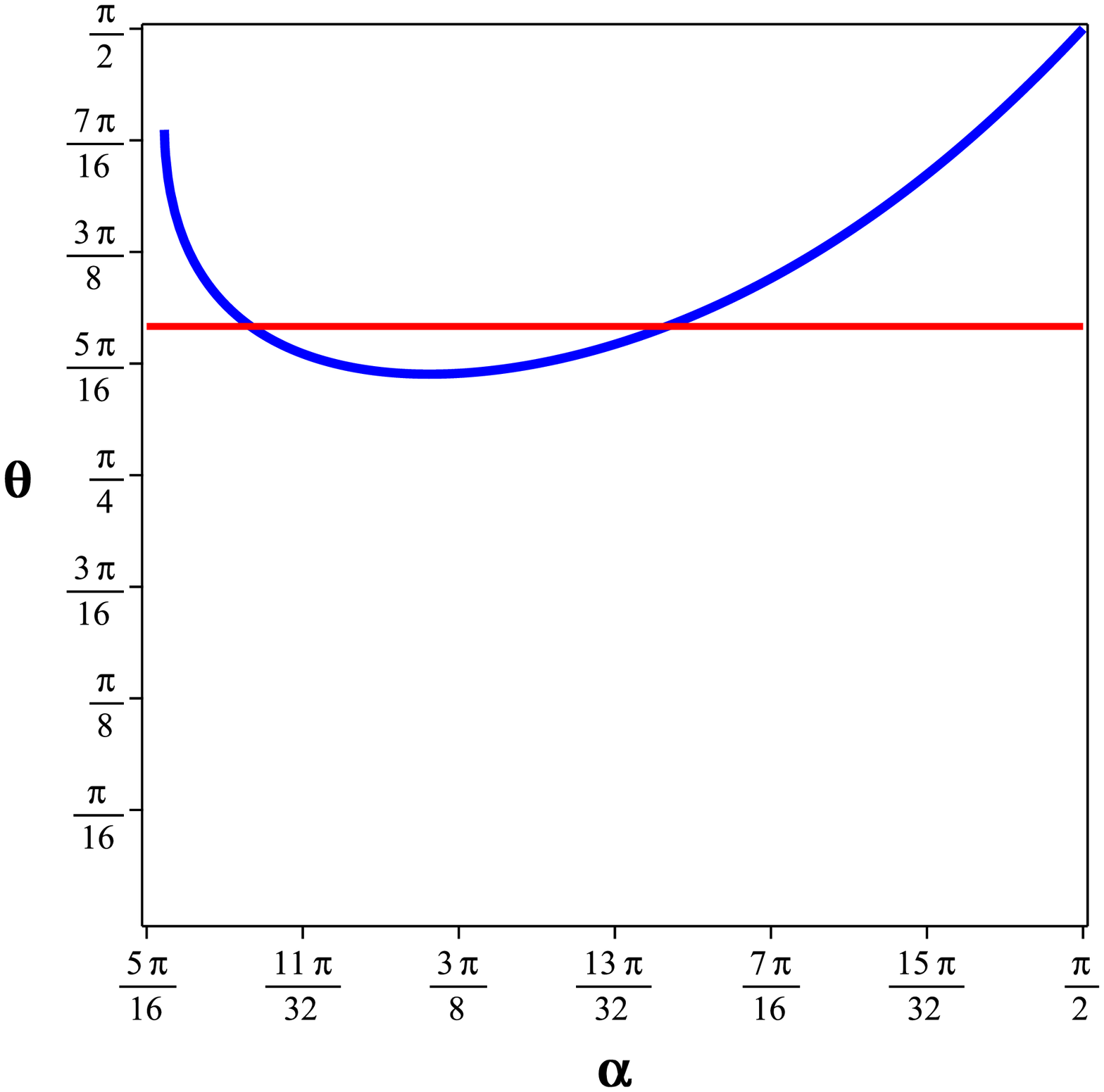} \\ a)}
\end{minipage}
\hspace{17mm}
\begin{minipage}[h]{0.4\textwidth}
\center{\includegraphics[width=0.9\textwidth]{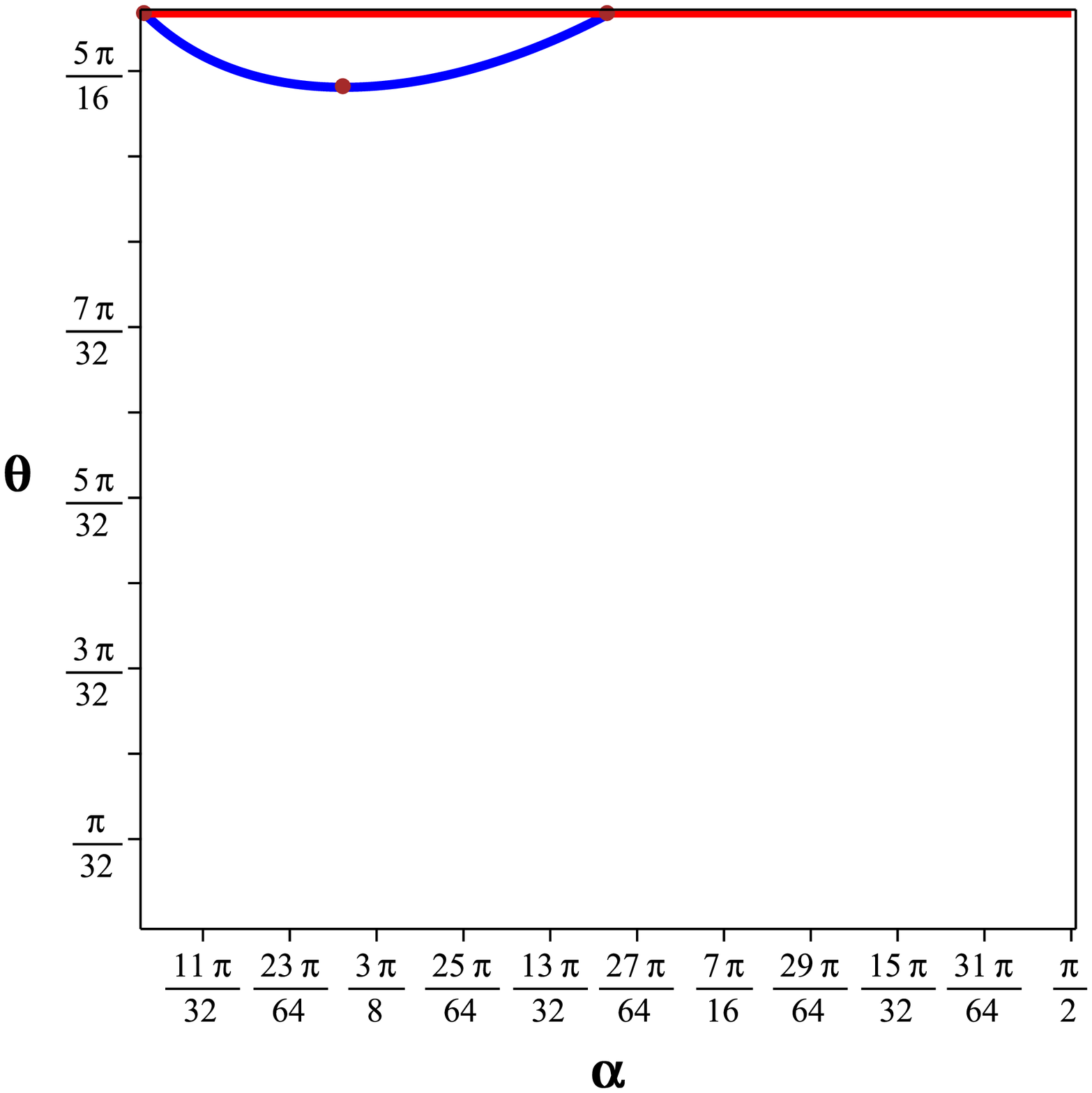} \\ b)}
\end{minipage}
\caption{
a) The function $\theta=g(\alpha)$;
b) The part of the graph of the function $\theta=g(\alpha)$ in the rectangle $[\pi/3,\pi/2]\times [0,\pi/3]$.
}
\label{Fig_difa}
\end{figure}

The minimum value of this function is
$\theta_0=\frac{\bigl(4\sqrt{37}-5\bigr)\sqrt{12 + 3\sqrt{37}}}
{\bigl(43-2\sqrt{37}\bigr)\sqrt{24 - 3\sqrt{37}}}
=0.9630621725\dots$ and it is reached at the point
$\alpha_0=\arccos\left(\frac{\sqrt{24 - 3\sqrt{37}}}{6}\right)=1.159593548\dots$.

It is easy to check that $\mathcal{D}(\alpha,\theta)>0$ for all points in $[\pi/3,\pi/2]\times [0,\pi/3]$
lying above the curve $\theta=g(\alpha)$. For all other points in $[\pi/3,\pi/2]$ that are not on the graph, we get $\mathcal{D}(\alpha,\theta)<0$.

The above arguments lead to the following observation.
If we have the pentagon with minimum perimeter and with the parameters $(\alpha,\beta,\gamma)$ such that $\cos(\alpha)+\cos(\beta)>1/2$
(that is an unconstrained minimum),
then $\theta = g(\alpha)$, $\theta \geq \theta_0$, and $\alpha \leq \alpha_0$.
Indeed, the sign of $\frac{\partial \mathcal{B}}{\partial \alpha}$ coincides with the sign of $\mathcal{D}(\alpha,\theta)$.
Further on, for a fixed $\theta$ and for $\alpha$ on the interval $[\pi/3,\pi/2]$ we see that $\mathcal{D}(\alpha,\theta)<0$ for $\theta <\theta_0$.
Hence, the point with such $\theta$ cannot be the minimum point for $\mathcal{B}$ if $\alpha< \pi/2$.
On the other hand, if $\theta \in [\theta_0, \pi/3]$, then we can have a minimum point for $\mathcal{B}$ only if $\theta=g(\alpha)$, where $\alpha \leq \alpha_0)$.
Indeed, in such points, $\frac{\partial \mathcal{B}}{\partial \alpha}$ changes the sign from ``--'' to ``+'' as $\alpha$ increases
(if $\alpha >\alpha_0$, then $\frac{\partial \mathcal{B}}{\partial \alpha}$ changes the sign from ``+'' to ``--'' and could not be the minimum point).

Note that the set $\Omega$ and the function $\mathcal{B}(\alpha,\beta,\theta)$ are symmetric with respect to the permutation of $\alpha$  and $\beta$.
Therefore, for a minimum point $(\alpha,\beta,\gamma)$ of the function $\mathcal{B}(\alpha,\beta,\theta)$ such that $\cos(\alpha)+\cos(\beta)>1/2$,
we get also the following:
$\theta = g(\beta)$, $\theta \geq \theta_0$, and $\beta \leq \alpha_0$. Since $g$ is (strictly) decreasing on the interval $[\pi/3,\alpha_0]$, we get that
$\alpha =\beta$. All the above reasoning leads to the following important

\begin{prop}\label{pr.perpen1}
If the point $(\alpha,\beta,\theta)$ is the point of absolute minimum value of the function $\mathcal{B}(\alpha,\beta,\theta)$ on the set $\Omega$,
then we have one of the following
possibilities:

1) $\beta=\alpha\leq \alpha_0=\arccos\left(\frac{\sqrt{24 - 3\sqrt{37}}}{6}\right)=1.159593548\dots$ and $\theta=g(\alpha)$;

2) $\cos(\alpha)+\cos(\beta)=1/2$.

\end{prop}

\begin{figure}[t]
\begin{minipage}[h]{0.4\textwidth}
\center{\includegraphics[width=0.9\textwidth]{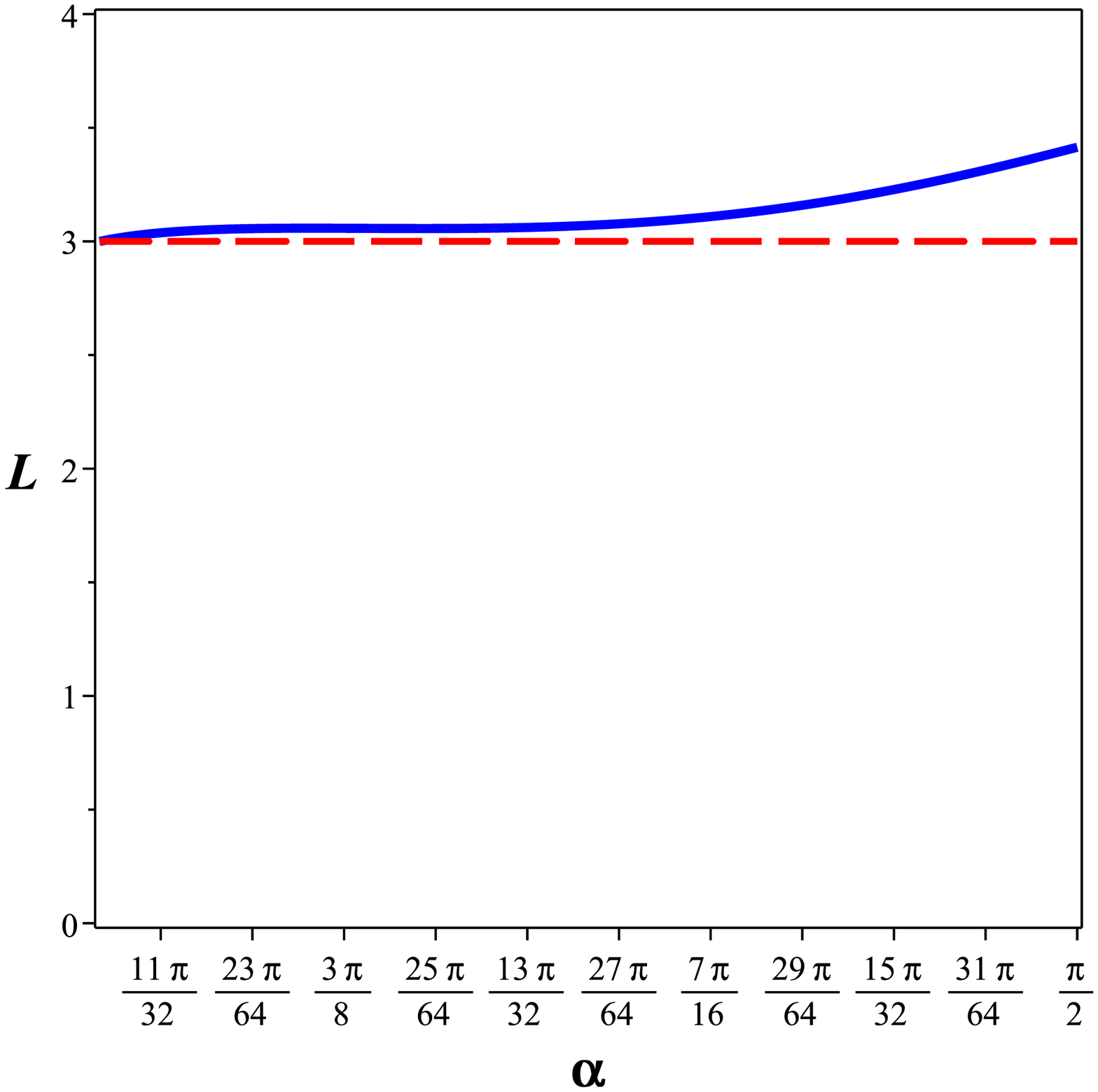} \\ a)}
\end{minipage}
\hspace{17mm}
\begin{minipage}[h]{0.4\textwidth}
\center{\includegraphics[width=0.9\textwidth]{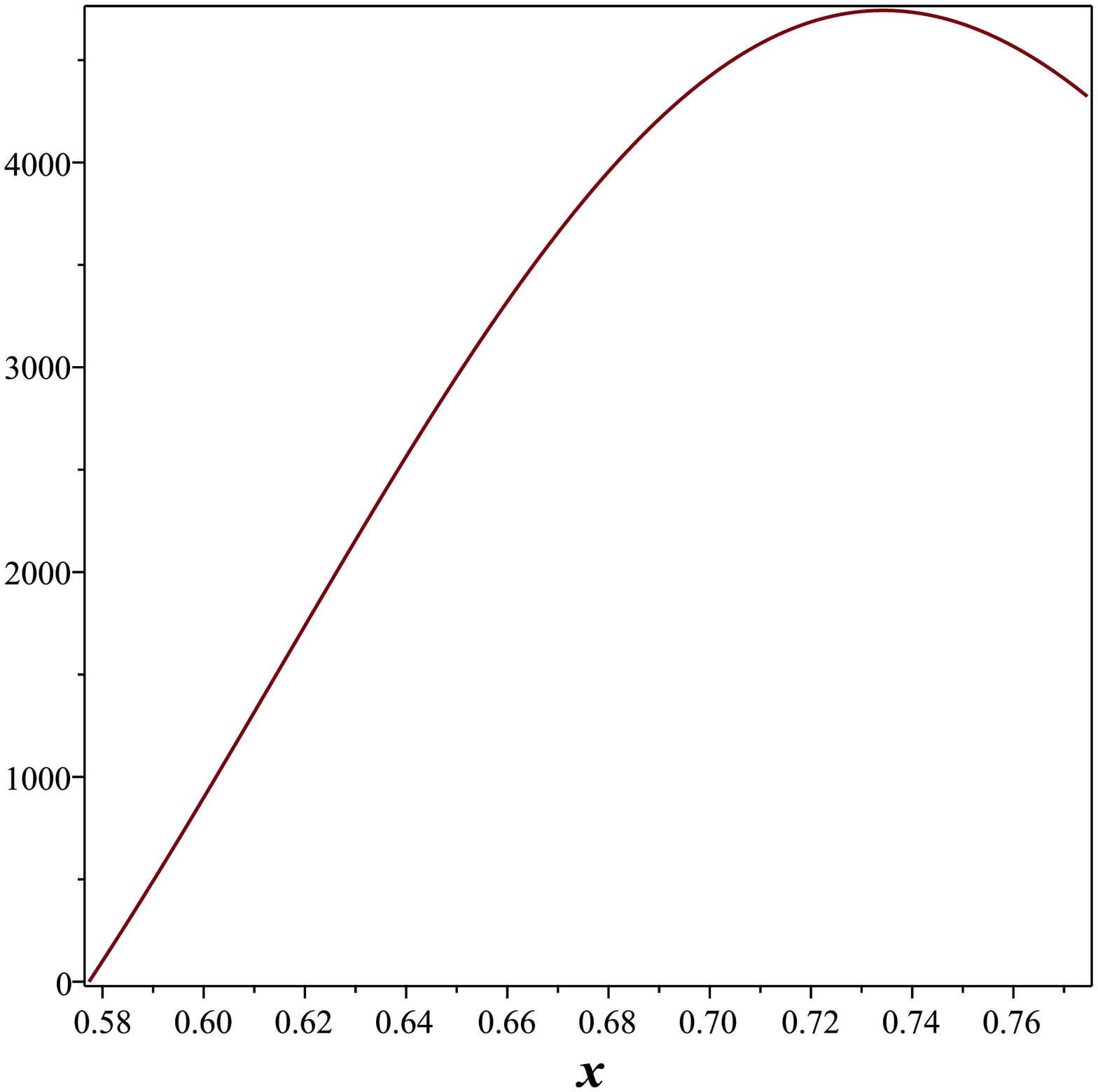} \\ b)}
\end{minipage}
\caption{
a) The graph of the function $\alpha \mapsto \mathcal{B}(\alpha,\alpha, g(\alpha))$ for $\alpha \in [\pi/3,\pi/2]$;
b) The graph of the function $V_3(x)$ for $x \in [1/\sqrt{3},\sqrt{3/5}]$.
}
\label{Fig_difaa}
\end{figure}

\begin{remark} If we take $\alpha=\beta=\pi/3$, then $\theta=g(\pi/3)=\pi/3$ and $\mathcal{B}(\pi/3,\pi/3,\pi/3)=3$.
In fact, the function $\alpha \mapsto \mathcal{B}(\alpha,\alpha, g(\alpha))$ takes values $>3$ for all $\alpha\in(\pi/3,\alpha_0]$, see Fig.~\ref{Fig_difaa}~a).
On the other hand, the explicit expression for this function is quite complicated to use it for proving the above observation. Hence, we need to find a different
approach.
Suitable tools in this case are working with polynomial ideals, the ability to eliminate some variables and to find the corresponding Gr\"{o}bner bases,
see, e.~g., \cite{CLOS2007}.
It is reasonable to perform all necessary calculations using some standard system of symbolic calculations.
\end{remark}
\smallskip

Now we are going to consider Case 1) in Proposition \ref{pr.perpen1} in detail.
Moreover, we consider all points $(\alpha,\beta,\gamma)\subset \Omega\subset [\pi/3,\pi/2]^3$ with $\beta=\alpha$. Since $\cos(\alpha)+\cos(\beta)\geq 1/2$, then
$\alpha \leq \arccos(1/4)$.
It is easy to write down the criticality condition for the function
$$
\mathcal{A}(\alpha,\alpha,\gamma)=4\cos(\alpha)+ 2\cos(\gamma)+2\sqrt{3+2\cos(2\alpha)+2\cos(2\gamma)+2\cos(2\alpha+2\gamma)}.
$$
We have
\begin{eqnarray*}
\sin(\alpha)+\frac{\sin(2\alpha)+\sin(2\alpha+2\gamma)}{\sqrt{3+2\cos(2\alpha)+2\cos(2\gamma)+2\cos(2\alpha+2\gamma)}}=0,\\
\sin(\gamma)+2\frac{\sin(2\gamma)+\sin(2\alpha+2\gamma)}{\sqrt{3+2\cos(2\alpha)+2\cos(2\gamma)+2\cos(2\alpha+2\gamma)}}=0.
\end{eqnarray*}
The above equations imply the following:
\begin{eqnarray*}
(3+2\cos(2\alpha)+2\cos(2\gamma)+2\cos(2\alpha+2\gamma))\sin(\alpha)^2&=&(\sin(2\gamma)+\sin(2\alpha+2\gamma))^2,\\
(3+2\cos(2\alpha)+2\cos(2\gamma)+2\cos(2\alpha+2\gamma))\sin(\gamma)^2&=&4(\sin(2\gamma)+\sin(2\alpha+2\gamma))^2.
\end{eqnarray*}
To study this system, we use the substitution
$\alpha=2\arctan(x)$, $\gamma=\arctan(z)$, which implies
\begin{eqnarray*}
V_1(x,z):=(2x^2z+xz^2+x-2z)\\
\times(2x^6z+3x^5z^2-5x^5-14x^4z-10x^3z^2+6x^3+14x^2z+3xz^2-5x-2z)=0,\\
V_2(x,z):=x^8z^4-55x^8z^2-112x^7z^3-60x^6z^4+128x^7z+388x^6z^2\\
+400x^5z^3+134x^4z^4-64x^6-384x^5z-650x^4z^2-400x^3z^3-60x^2z^4\\
+128x^4+384x^3z+388x^2z^2+112xz^3+z^4-64x^2-128xz-55z^2=0,
\end{eqnarray*}
Since $\alpha \in [\pi/3,\arccos(1/4)]$ and $\gamma \in [\pi/3,\pi/2)$, we get $x \in [1/\sqrt{3},\sqrt{3/5}]$ and $z\geq \sqrt{3}$.

We may eliminate $z$ from $V_1(x,z)=V_2(x,z)=0$ and obtain the equation $V_3(x)=0$, where
\begin{eqnarray*}
V_3(x)=x(x-1)(x+1)(x^2+1)(63-1028x^2+1914x^4-1028x^6+63x^8)\\
\times(3x^2-1)(x^2-3)(39-708x^2+1322x^4-708x^6+39x^8).
\end{eqnarray*}
It is easy to check that a unique root of $V_3(x)$ on the interval $[1/\sqrt{3},\sqrt{3/5}]$ is $x=1/\sqrt{3}$, corresponding to $\beta=\alpha=\pi/3$,
see Fig. \ref{Fig_difaa}~b). We see that $\mathcal{A}(\pi/3,\pi/3,\pi/3)=0$
(the pentagon $P$ degenerates into a regular triangle) and
$$
\mathcal{A}(\pi/3,\pi/3,\gamma)=2 + 2\cos(\gamma) + 2\sqrt{2 + \cos(2\gamma) - \sin(2\gamma)\sqrt{3}}>3
$$
for all $\gamma\in (\pi/3,\pi/2)$.

Hence, we get the following

\begin{prop}\label{pr.perpen2}
In Case 1) of Proposition \ref{pr.perpen1}
we have exactly one {\rm(}degenerate{\rm)} pentagon, corresponding to $(\alpha,\beta,\gamma)=(\pi/3,\pi/3,\pi/3)$,
with perimeter $3$. All other pentagons in Case 1) have perimeter greater that $3$.
\end{prop}

\smallskip

Now we are going to consider Case 2) in Proposition \ref{pr.perpen1} in details.
The boundary of the set
\begin{equation}\label{eq.setpern}
\Omega_1= \left\{(\alpha,\beta,\gamma) \in [\pi/3,\pi/2]^3 \subset \mathbb{R}^3 \,|\, \cos(\alpha)+\cos(\beta)= 1/2 \right\}
\end{equation}
consists of points of the following types:
$(\alpha,\beta,\gamma)=(\alpha,\beta,\pi/3)$ and $(\alpha,\beta,\gamma)=(\alpha,\beta,\pi/2)$, with $\cos(\alpha)+\cos(\beta)= 1/2$;
$(\alpha,\beta,\gamma)=(\pi/3,\pi/2,\gamma)$ and $(\alpha,\beta,\gamma)=(\pi/2,\pi/3,\gamma)$, with $\gamma \in (\pi/3,\pi/2)$.
Simple arguments lead to the following

\begin{prop}\label{pr.perpen3p}
If a point $(\alpha,\beta,\gamma)$ is in the boundary of the set $\Omega_1$ in Case 2) of Proposition \ref{pr.perpen1},
then $L(P)=\mathcal{A}(\alpha,\beta,\gamma)\geq 3$, whereas
$L(P)=\mathcal{A}(\alpha,\beta,\gamma)=3$ exactly for the points $(\pi/3,\pi/2,\pi/3)$, $(\pi/2,\pi/3,\pi/3)$,
and $(\alpha,\beta,\pi/2)$ with $\cos(\alpha)+\cos(\beta)=1/2$.
\end{prop}

We illustrate this proposition and Case 2) in general
by showing the graphs of the functions $\mathcal{F}(u,\gamma)= \mathcal{C}(u,{1}/{2}-v,\gamma)$
and $\mathcal{G}(\gamma)= \mathcal{C}(1/4,1/4,\gamma)=\mathcal{F}(1/4,\gamma)$, see Fig. \ref{Fig_rest} (recall that
$\cos(\alpha)+\cos(\beta)=1/2$ is equivalent to $u+v=1/2$).
The points $(u,v,\gamma)=(0,1/2,\pi/3)$ and $(u,v,\gamma)=(1/2,0,\pi/3)$ correspond to
$(\alpha,\beta,\gamma)=(\pi/2,\pi/3,\pi/3)$ and $(\alpha,\beta,\gamma)=(\pi/3,\pi/2,\pi/3)$, when the pentagon $P$ degenerates into a regular triangle with perimeter $3$.
The points $(u,v,\gamma)=(u,1/2-u,\pi/2)$, $u\in [0,1/2]$, correspond to the cases with $F=E$ and $|AB|+|BC|=1$ ($P$ degenerates to some quadrangles with perimeter $3$).
The points $(u,v,\gamma)=(u,1/2-u,\pi/3)$ for $u\in (0,1/2)$ correspond to the pentagons $P$ with perimeter $>3$.
\smallskip

\begin{figure}[t]
\begin{minipage}[h]{0.43\textwidth}
\center{\includegraphics[width=1.2\textwidth, trim=0mm 9mm 0mm 19mm, clip]{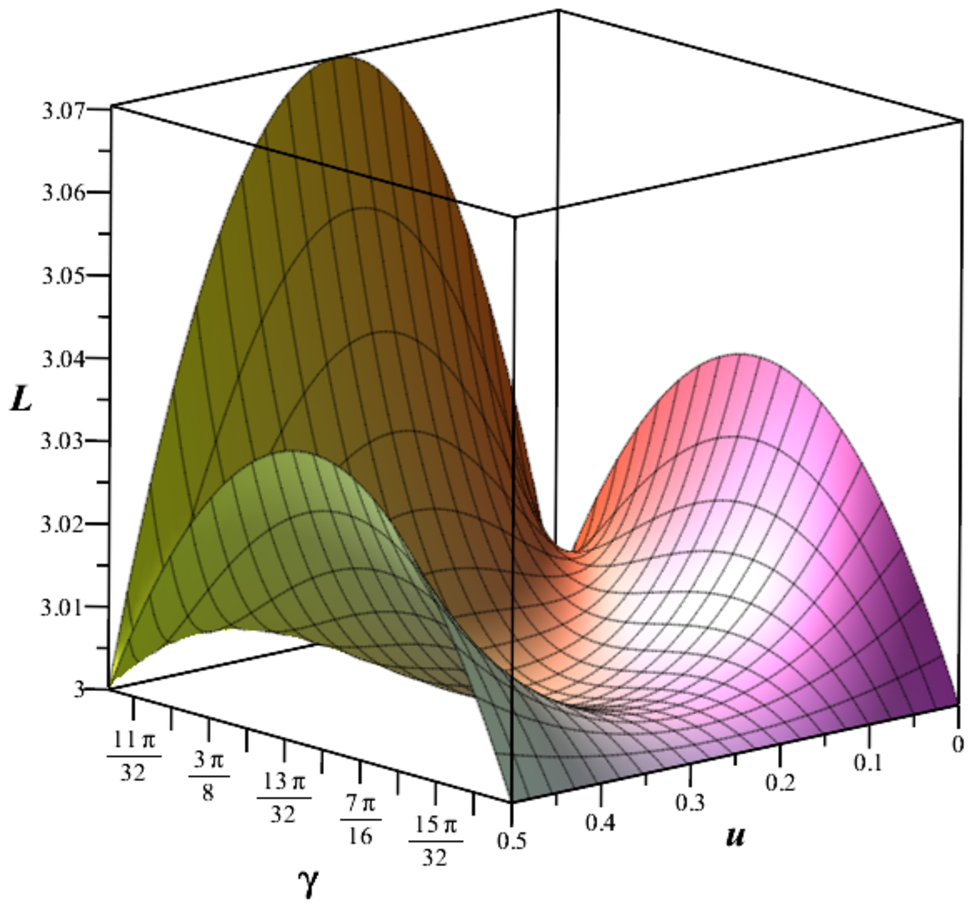} \\ a)}
\end{minipage}
\hspace{7mm}
\begin{minipage}[h]{0.43\textwidth}
\center{\includegraphics[width=0.9\textwidth, trim=0mm 0mm 0mm 0mm, clip]{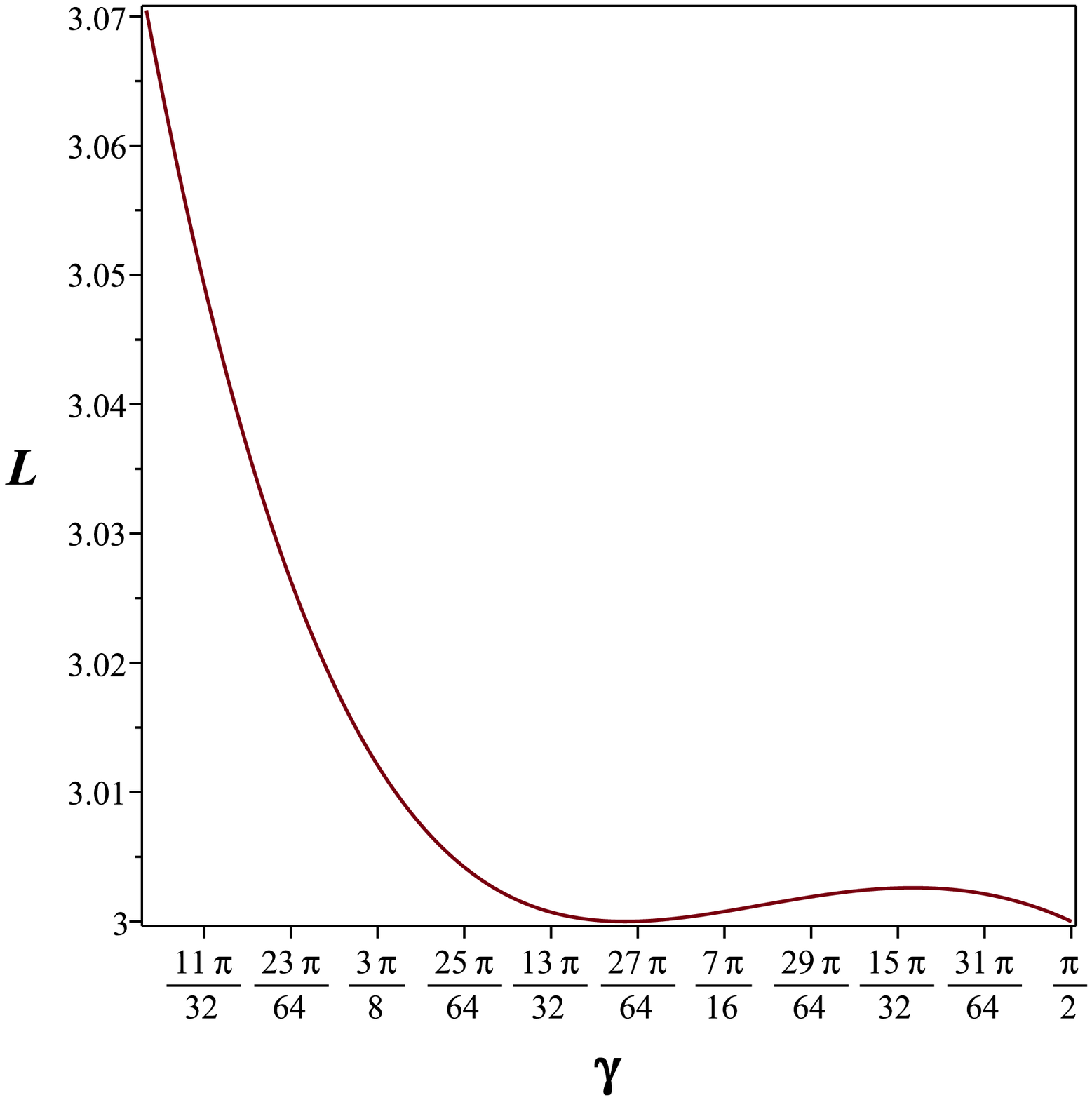} \\ b)}
\end{minipage}
\caption{
a) The graph of  $\mathcal{F}(u,\gamma)$;
b) The graph of $\mathcal{G}(\gamma)$.
}
\label{Fig_rest}
\end{figure}

Now we are going to study points in the interior of the set $\Omega_1$.
According to \eqref{eq.per1} and~\eqref{eq.per2}, we have

\begin{equation}\label{eg.perr1}
L(P)=\mathcal{A}(\alpha,\beta,\gamma)=2 \bigl(\cos(\alpha)+ \cos (\beta)+ \cos(\gamma)\Bigr)+\sqrt{f(\alpha)}+\sqrt{f(\beta)},
\end{equation}
where
$$
f(x)=3+2\cos(2x)+\cos(2\gamma)+\cos(2x+2\gamma)=1+4\cos^2(x)+4\cos(x)\cos(x+2\gamma).
$$
Since $\cos(\alpha)+\cos(\beta)=1/2$, $\alpha, \beta \in [\pi/3,\pi/2]$, then
$\alpha=\pi/2$ is equivalent to $\beta=\pi/3$ as well as $\alpha=\pi/3$ is equivalent to $\beta=\pi/2$.
Note also that
$\mathcal{A}(\pi/2,\pi/3,\gamma)=\mathcal{A}(\pi/3,\pi/2,\gamma)\geq 3$ for all $\gamma$ (since $P$ contains a regular triangle with perimeter $3$).
In what follows, we assume that $\alpha, \beta \in (\pi/3,\pi/2)$.
Further, we will use the condition to be a critical point for any point with minimal value for
$\mathcal{A}(\alpha,\beta,\gamma)$. Obviously, this condition has the following form (recall that the constraint condition is $\cos(\alpha)+\cos(\beta)=1/2$):

\begin{equation}\label{eq.crit1}
\frac{\partial \mathcal{A}}{\partial \gamma}(\alpha, \beta, \gamma)=0, \quad
\frac{\partial \mathcal{A}}{\partial \alpha}(\alpha, \beta, \gamma)\cdot \sin(\beta)=\frac{\partial \mathcal{A}}{\partial \beta}(\alpha, \beta, \gamma)\cdot \sin(\alpha).
\end{equation}

This system of equations can be written as follows:

\begin{eqnarray*}
\sin(\gamma)+\frac{\sin(2\gamma)+\sin(2\alpha+2\gamma)}{\sqrt{f(\alpha)}}+\frac{\sin(2\gamma)+\sin(2\beta+2\gamma)}{\sqrt{f(\beta)}}&=&0,\\
\frac{\sin(2\alpha)+\sin(2\alpha+2\gamma)}{\sqrt{f(\alpha)}\sin(\alpha)}
-\frac{\sin(2\beta)+\sin(2\beta+2\gamma)}{\sqrt{f(\beta)}\sin(\beta)}&=&0.
\end{eqnarray*}
From these equations, we obtain explicit expressions for $\sqrt{f(\alpha)}$ and $\sqrt{f(\beta)}$. Further, since
$f(x)=3+2\cos(2x)+\cos(2\gamma)+\cos(2x+2\gamma)$, then
squaring each of the resulting expressions and equating to the explicit expression through the function $f$, we get two equations
in variables $\alpha, \beta, \gamma$. Then we obtain polynomial equations $U_1(x,y,z)=0$ and $U_2(x,y,z)=0$ using the following change of variables:
$$
\alpha=2\arctan(x), \quad \beta=2\arctan(y),\quad \gamma=\arctan(z).
$$
Since $\alpha,\beta,\gamma \in [\pi/3,\pi/2]$, then $x,y \in [1/\sqrt{3},1]$ and $z\geq \sqrt{3}$.

We do not present here explicit expressions for the polynomials $U_1(x,y,z)$ and $U_2(x,y,z)$ due to their cumbersomeness.
In the polynomial ideal generated by these two polynomials, there is the very useful polynomial
$$
z^{2}(1 + y^{2})^{2}(1 + x^{2})^{2}(x\,y + 1)(x - y)\times U_4(x,y,z),
$$
where
\begin{eqnarray*}
U_4(x,y,z)=
yz^{4} + 48z^{3}x^{5}y^{3} + 9z^{2}y + 72xyz + 9xz^{2} + xz^{4} - 64y^{3}x^{4} - 178z^{2}x^{3}y^{4} \\
- 64x^{3}y^{4}
+ 64x^{3}y^{2} + 64y^{3}x^{2} + 9y^{5}z^{2} + 2y^{3}z^{4} - 14y^{3}z^{2} - 16z^{3}y^{2} + y^{5}z^{4} \\
+ 2x^{3}z^{4}+ x^{5}z^{4}
- 14x^{3}z^{2} + 9x^{5}z^{2} - 16z^{3}x^{2} + 16z^{3}x^{4} + 16z^{3}y^{4} - z^{4}x^{5}y^{6} - z^{4}x^{5}y^{4} \\
- 2z^{4}x^{3}y^{6}
- z^{4}xy^{6} + 62z^{4}x^{3}y^{4} - z^{4}xy^{4} - 62z^{4}x^{3}y^{2} + z^{4}x^{5}y^{2} + z^{4}xy^{2} - z^{4}yx^{4} \\
- z^{4}yx^{6}
+ z^{4}yx^{2}
- z^{4}y^{5}x^{4} - 62z^{4}y^{3}x^{2} + 62
z^{4}y^{3}x^{4} + z^{4}y^{5}x^{2} - z^{4}y^{5}x^{6}
 - 2z^{4}y^{3}x^{6} \\
 - 8z^{3}x^{5}y^{5} + 48z^{3}x^{3}y^{5}
- 8z^{3}xy^{5} + 48z^{3}x^{3}y + 48z^{
3}xy^{3} - 32z^{3}x^{3}y^{3} - 8z^{3}x^{5}y\\
- 8z^{3}xy + 288z^{3}y^{4}x^{4}
+ 16z^{3}y^{6}x^{2} - 16z^{3}y^{6}x^{4}
 + 288z^{3}x^{2}y^{2} - 288z^{3}x^{4}y^{2} \\
 + 16z^{3}x^{6}y^{2} - 288z^{3}x^{2}y^{4}
- 16z^{3}y^{4}x^{6} - 9z^{2}x^{5}y^{6}
 + 151z^{2}x^{5}y^{4} + 14z^{2}x^{3}y^{6} \\
 - 9z^{2}xy^{6} + 151z^{2}xy^{4} + 178z^{2}x^{3}y^{2}
- 151z^{2}x^{5}y^{2} - 151z^{2}xy^{2} +151z^{2}yx^{4} \\
- 9z^{2}yx^{6} - 151z^{2}yx^{2}+ 151z^{2}y^{5}x^{4}
+ 178z^{2}y^{3}x^{2} - 178z^{2}y^{3}x^{4}- 151z^{2}y^{5}x^{2}\\
- 9z^{2}y^{5}x^{6} + 14z^{2}y^{3}x^{6} + 72zx^{5}y^{5}
- 112zx^{3}y^{5} + 72zxy^{5} - 112zx^{3}y - 112zxy^{3} \\
+ 32zx^{3}y^{3} - 112zx^{5}y^{3} + 72zx^{5}y
+ 64zy^{4}x^{4} + 64zx^{2}y^{2} - 64zx^{4}y^{2} - 64zx^{2}y^{4}.
\end{eqnarray*}

We see that there are two possibilities: $y=x$ or $U_4(x,y,z)=0$.
Under the same change of variables the condition $\cos(\alpha)+\cos(\beta)=1/2$ becomes as follows:
$$
U_3(x,y):=3-x^2-y^2-5x^2y^2=0.
$$
\smallskip

{\bf Let us suppose that $y=x$}. In this case the equation $0=U_3(x,y)=3 - 2x^{2} - 5x^{4}$ implies $x=\sqrt{\frac{3}{5}}$.
Further, $U_1(x,x,z)=U_2(x,x,z)=0$ implies
$$
(z_1 - 1)(15 z_1^{3} - 45 z_1^{2} + 5z_1+ 1)=0,
$$
where $z_1=z/\sqrt{15}$.
The latter equation has four real solutions ($-0.1024023606\dots$, $0.2263621549\dots$, $1$, $2.876040206\dots$).
Since $z\geq \sqrt{3}$, then $z_1 \geq 1/\sqrt{5}=0.4472135954\dots$,
hence we have only two critical points.
The corresponding two pentagons have perimeter $3$ and $3.002605955\dots$ respectively (see Fig. \ref{Fig_rest}~b)).

Therefore, we unexpectedly found {\bf a non-degenerate pentagon $P$ with perimeter $3$}. It corresponds to the values
$y=x=\sqrt{\frac{3}{5}}$ and $z=\sqrt{15}$ or, equivalently, to the values $\alpha=\beta=\gamma=\arccos(1/4)$, see Fig. \ref{Fig_arch}~b).
\smallskip

{\bf Now, let us suppose that $U_4(x,y,z)=0$.}
Eliminating the variables $z$ from two polynomials $U_1(x,y,z)$ and $U_4(x,y,z)$, we get the polynomial $U_5(x,y)\cdot U_6(x,y)$,
where the latter one is too large for being explicitly presented here, and
\begin{eqnarray*}
U_5(x,y)=xy(3x^2-1)(3y^2-1)(x^2-3)(y^2-3)(1+x^2)(1+y^2)\\
\times (x+y)(x-y)(xy+1)(xy+y-x+1)(xy-y+x+1).
\end{eqnarray*}

Solving $U_3(x,y)=3-x^2-y^2-5x^2y^2=0$ together with $U_5(x,y)=0$, we get $11$ solutions. Only three of these solutions satisfy $x,y \in[1/\sqrt{3},1]$:
$(x,y)=(1/\sqrt{3},1)$, $(x,y)=(1,1/\sqrt{3})$, and $(x,y)=(\sqrt{3/5},\sqrt{3/5})$.
The latter one we have obtained in the previous stage (the perimeter of the corresponding pentagon is 3).
Other two solutions correspond to $(\alpha,\beta)=(\pi/3,\pi/2)$ and $(\alpha,\beta)=(\pi/2,\pi/3)$, where we have $L(P) \geq 3$.
\smallskip

Further, from the equations $U_3(x,y)=3-x^2-y^2-5x^2y^2=0$ and $U_6(x,y)=0$ (assuming that $y\neq x$)
we may eliminate $x$ and get the equation $U_7(y)=0$, where the polynomial $U_7$ has degree $104$.
Only $6$ of its roots are in the interval $[1/\sqrt{3},1]$. This observation gives the following 6 solutions for $(x,y)$:
\begin{eqnarray*}
(0.6289625043\dots, 0.9351779532\dots), \quad
(0.6517483364\dots, 0.9079467534\dots),\\
(0.8683688602\dots, 0.6861590202\dots), \quad
(0.6861590202\dots, 0.8683688602\dots),\\
(0.9079467534\dots, 0.6517483364\dots), \quad
(0.9351779532\dots, 0.6289625043\dots).\,
\end{eqnarray*}
When we compute the corresponding values of $z$, we see that only in the first and the last cases we get $z\geq \sqrt{3}$.
Hence, we have the following two variants:
\begin{eqnarray*}
(x,y,z)&=&(0.6289625043\dots, 0.9351779532\dots, 2.242807541\dots),\\
(x,y,z)&=&(0.9351779532\dots, 0.6289625043\dots, 2.242807541\dots).
\end{eqnarray*}
The two corresponding pentagons have perimeter $3.008459178\dots$.
Hence, we get the following

\begin{prop}\label{pr.perpen3}
In the interior of the set $\Omega_1$ in Case 2) of Proposition \ref{pr.perpen1},
we have exactly one pentagon, corresponding to $(\alpha,\beta,\gamma)=(\arccos(1/4),\arccos(1/4),\arccos(1/4))$,
with perimeter $3$. All other pentagons in the interior of  $\Omega_1$ have perimeter greater than~$3$.
\end{prop}

\begin{remark}
It is also possible to prove Proposition \ref{pr.perpen3}
by explicitly calculating the Gr\"{o}bner basis of the polynomial ideal generated by polynomials $U_1(x,y,z)$, $U_2(x,y,z)$, and $U_3(x,y)$
(using some system of analytical computations).
\end{remark}

\section{The proofs of the main results}\label{sect.3}

In the Introduction we have obtained the proof of Theorem \ref{th.minim} for $n=3$ and the fact that
for any $n$ there are $n$-gons $P$ with the $\Delta$ property and with the perimeter $L(P)$ arbitrarily close to $3$.
Moreover, a regular triangle with unit side is the unique triangle with the $\Delta$ property and with perimeter $3$.

In what follows, we will prove the inequality $L(P)\geq 3$ for any polygon $P$ with the $\Delta$ property.
Moreover, we will prove that for any $n\geq 4$, there is no such $n$-gon with perimeter $3$.
\smallskip

Let us fix $n\geq 3$ and any $n$-gon $A_1A_2A_3\dots A_{n-1}A_n$ with the $\Delta$ property. We consider two possibilities:

Case 1) There is a side of the polygon $P$ whose length is at least $1$;

Case 2) All sides of the polygon $P$ have length $<1$.
\smallskip

It is clear that Case 1) is simple.

\begin{prop}\label{pr.doc1}
If a polygon $P$ has the $\Delta$ property and one of its sides has length $\geq 1$, then $L(P) \geq 3$.
Moreover, $L(P)=3$ if and only if $P$ is a regular triangle with side $1$.
\end{prop}

\begin{proof}
Without loss of generality we may assume that $|A_1A_2|\geq 1$. By the $\Delta$ property, there is a point $B\in P$ such that
$|A_1B|=|A_2B|=1$. Since the triangle $A_1A_2B$ is a subset of $P$, $|A_1B|=|A_2B|=1$, and $|A_1A_2|\geq 1$, then
$L(P)\geq L(\triangle A_1A_2B)=|A_1B|+|A_2B|+|A_1A_2|\geq 3$ (see Proposition \ref{monotper}).
Note that $L(P)$=3 implies $L(\triangle A_1A_2B)\leq 3$, hence, $|A_1A_2|=1$, $L(\triangle A_1A_2B)=3$,
and $P=\triangle A_1A_2B$ (by the same Proposition \ref{monotper}). Therefore, $P$ is a regular triangle with side $1$.
\end{proof}
\smallskip

Now, we are going to study Case 2).

\begin{prop}\label{pr.doc2}
If a polygon $P$ has the $\Delta$ property and any of its sides has length $<1$, then $L(P)> 3$.
\end{prop}

\begin{proof}
At first, we consider the case when $P$ has no pair of parallel sides.
Let us consider the difference body $D(P)=P+(-P)$. Recall that $D(P)$ is a centrally symmetric polygon with the center of symmetry
at the origin $O$. By Proposition \ref{pr.syssym}, $D(P)$ has the $\Delta^s$ property, and its perimeter $L\bigl(D(P)\bigr)$ is equal to $2\cdot L(P)$.

Moreover, for any side  of $P$ there are exactly two sides of $D(P)$ that are parallel to the first one and have the same length.
In particular, all sides of $D(P)$ have length $<1$.

By Theorem \ref{th.centsym}, we get that $L(D(P)) >6$. Hence, $L(P)=\frac{1}{2} L\bigl(D(P)\bigr)>3$.
\smallskip

Now let us consider $P$ of the general type (it can have several pairs of mutually parallel sides).
Let us consider the sequence $\{P_n\}_{n\in \mathbb{N}}$ such that

1) every $P_n$ has the $\Delta^s$ property;

2) every sides of any polygon $P_n$ has length $<1$;

3) $P_n \rightarrow P$ in the Hausdorff metric as $n \rightarrow \infty$.

As it is shown earlier, we have the inequalities $L(P_n) >3$. Passing to the limit, we get $L(P) \geq 3$.

Now, it suffices to prove that the equality $L(P)=3$ is impossible. We suppose that $L(P)=3$ and
consider the sequence of the corresponding difference bodies $\{D(P_n)\}_{n \in \mathbb{N}}$.
All sides of $D(P_n)$ are shorter than $1$ for every $n$ (due to the fact that $P_n$ has the same property and has not pairwise parallel sides).
By Proposition \ref{pr.syssym},
$$
L\bigl(D(P_n)\bigr)=2L(P_n)\rightarrow 2 L(P)=L\bigl(D(P)\bigr)=6
$$
as $n \rightarrow \infty$.
By Proposition \ref{pr.centsym}, we get that $D(P)$ is a regular hexagon with side of length $1$.

Finally, since $P$ has the $\Delta$ property, then
$P$ is a regular triangle with unit side length by Proposition \ref{pr.proreghex}, which is impossible.
The proposition is completely proved.
\end{proof}
\smallskip

{\bf The proof of Theorem \ref{th.minim}} follows immediately from Propositions \ref{pr.doc1} and \ref{pr.doc2}.
\medskip

Now we are going to prove our second main result.
\smallskip

\begin{figure}[t]
\begin{minipage}[h]{0.45\textwidth}\hspace{1mm}\\
\center{\includegraphics[width=0.8\textwidth, trim=0mm 0mm 0mm 0mm, clip]{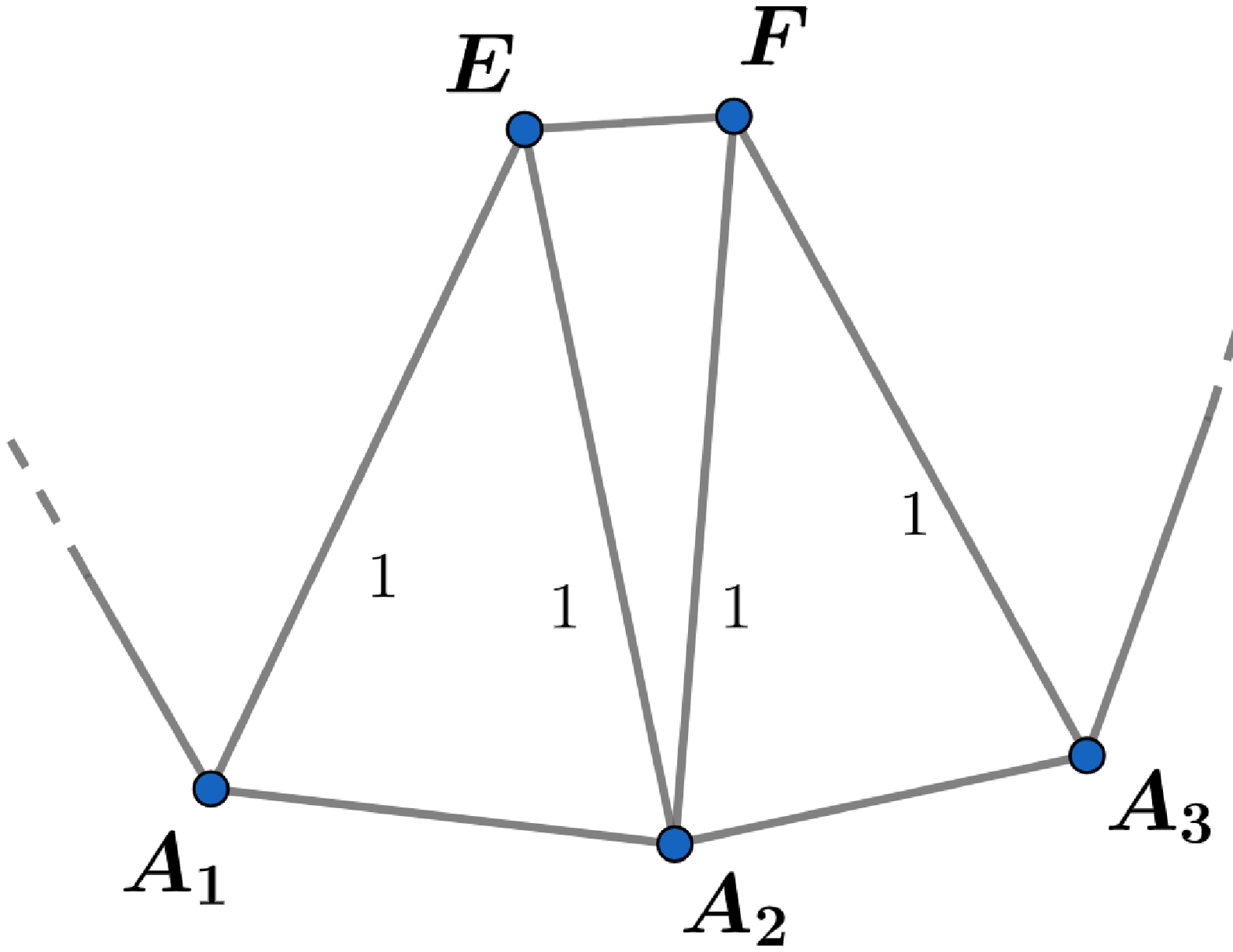} \\ \hspace{1mm}\\ a)}
\end{minipage}
\hspace{1mm}
\begin{minipage}[h]{0.45\textwidth}
\center{\includegraphics[width=0.9\textwidth, trim=0mm 0mm 0mm 0mm, clip]{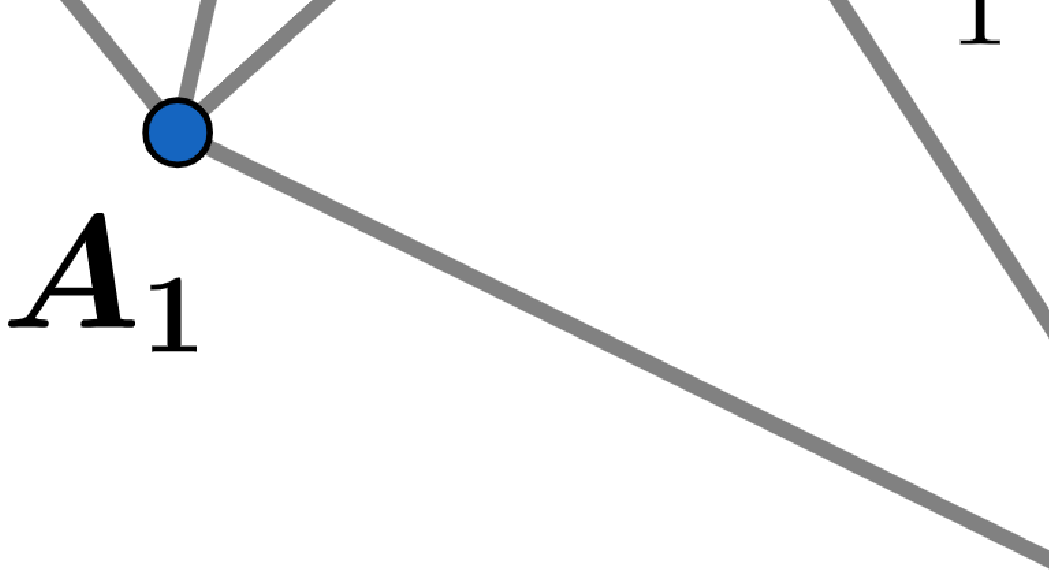} \\ b)}
\end{minipage}
\caption{
a) The first case for the polygon $P_1$;
b) The second case for the polygon $P_1$.
}
\label{Fig_uniq3}
\end{figure}

\begin{proof}[Proof of Theorem \ref{th.minim.2side}]
Since the sides $A_1A_2$ and $A_2A_3$ have struts, there are points $E,F \in P$ such that $|A_1E|=|A_2E|=1$ and $|A_2F|=|A_3F|=1$.
Let $P_1$ be the convex hull of the points $A_1, A_2, A_3, E, F$.
There are two cases: 1) the segments $[A_1,E]$ and $[A_3,F]$ are on the boundary of $P_1$ (see Fig.~\ref{Fig_uniq3}~a))
and 2) the segments $[A_1,E]$ and $[A_3,F]$ have a common point in the interior of $P_1$ (see Fig.~\ref{Fig_uniq3}~b)).
\smallskip

In case first case, we see that
$$
L(P_1)=|A_1E|+|A_3F|+\bigl(|A_1A_2|+|A_2A_3|\bigr)+|EF|\geq 3+|EF|.
$$
By Proposition \ref{monotper}, $L(P)\geq L(P_1)\geq 3$. It is clear that $L(P)=3$ only if
$P=P_1$, $E=F$, $|A_1A_2|+|A_1A_3|=1$. On the other hand, such a quadrangle $P=A_1A_2A_3E$ has not the $\Delta$ property.
For instance, there is no point $B\in P$ such that $|A_1B|=|EB|=1$.
\smallskip

In the second case, we see that $P_1$ satisfies all assumptions of Theorem \ref{th.perpent}.
From this theorem and Proposition \ref{monotper}, we get $L(P)\geq L(P_1)\geq 3$.
On the other hand, $L(P)=3$ implies $L(P_1)=3$ and $P=P_1$.
\end{proof}
\medskip

We hope that the results presented in this paper will serve as a basis for related studies.
It would also be interesting to obtain some analogues of Theorems \ref{th.minim}, \ref{th.minim.2side}, and \ref{th.perpent}
taking, instead of the
perimeter $L(P)$, the area $S(P)$ of the polygon $P$ into consideration.

\vspace{15mm}

\end{document}